\documentclass[11pt]{amsart}

\usepackage{setspace} 
\usepackage{amssymb}
\usepackage[dvips]{color}
\usepackage{amsmath}
\usepackage{amsthm}
\usepackage[a4paper]{geometry}
\usepackage{amsfonts}
\usepackage[all]{xypic} 
\usepackage{bbm}
\usepackage{xcolor}
\usepackage[normalem]{ulem}

\usepackage{color}
\usepackage[colorlinks=true,linkcolor=blue,citecolor=blue,pdfpagelabels=false]{hyperref}
\usepackage{cleveref}
\usepackage{url}
\usepackage{enumitem}
\usepackage{hyperref}

\usepackage[colorinlistoftodos]{todonotes}

\allowdisplaybreaks[4] 
\newtheoremstyle{myremark}     {10pt}{10pt}{}{}{\bfseries}{.}{.5em}{}

\newtheorem{thm}{Theorem}[section]

\newtheorem{lem}[thm]{Lemma}
\newtheorem{prop}[thm]{Proposition}
\theoremstyle{definition}
\newtheorem{defn}[thm]{Definition}

\newtheorem{rem}[thm]{Remark}
\numberwithin{equation}{section}

\def\C{\mathbb C}

\def\N{\mathbb N}

\def\T{\mathbb T}
\def\Z{\mathbb Z}

\newcommand{\CB}{\mathcal{B}}

\newcommand{\CH}{\mathcal{H}}

\newcommand{\CM}{\mathcal{M}}

\newcommand{\CP}{\mathcal{P}}

\newcommand{\CS}{\mathcal{S}}

\newcommand{\CU}{\mathcal{U}}

\newcommand{\eps}{\varepsilon}

\newcommand{\abs}[1]{\left\vert#1\right\vert}

\newcommand{\norm}[1]{\left\Vert#1\right\Vert}

\begin{document}

	\title[Ergodic Theory along Reiter sequence]{ Pointwise Ergodic Theorems Along Reiter Sequences on Locally Compact Groups}
	
	\author[P. Bikram]{Panchugopal Bikram$^{1, 2}$}
\address{$^{1}$School of Mathematical Sciences,
National Institute of Science Education and Research,  Bhubaneswar, Odisha, India  \newline  $^{2}$Homi Bhabha National Institute, Training School Complex, Anushakti Nagar, Mumbai 400094, India}
	
\author[S.Kundu]{Sudipta Kundu$^{1, 2}$}
	
\author[Hariharan G.]{Hariharan G$^{1, 2}$}
	
	\email{bikram@niser.ac.in}  \email{sudipta.kundu@niser.ac.in}
    	\email{hariharan.g@niser.ac.in} 

\author[D. Saha]{Diptesh Saha }
	\address{Institute of Mathematics of the Polish Academy of Sciences,  ul. Sniadeckich 8, 
		00–656 Warszawa, Poland}
	\email{dptshs@gmail.com}
        

	\begin{abstract}
Let $G$ be a locally compact, second countable group admitting a Reiter sequence, and suppose that $G$ acts on a von Neumann algebra $\CM$ equipped with a semifinite trace $\tau$ via trace-preserving automorphisms. In this paper, we prove a maximal inequality for ergodic averages along such Reiter sequences. Our approach relies on approximating the Reiter sequence by a linear combination of characteristic functions of F\o lner sequences using the layer-cake representation. Furthermore, we establish the bilateral almost uniform convergence of these ergodic averages. As an application, we deduce a non-commutative Wiener-Wintner type theorem along Reiter sequences.

    	\end{abstract}

    \maketitle

\section{Introduction}
\subsection{Background and methodology:} 
Ergodic theory and operator algebras have been closely intertwined since the seminal work of von Neumann. A central theme of ergodic theory is the study of the long-term behaviour of dynamical systems through the convergence of suitable ergodic averages. The subject originated in the 1930s with the fundamental results of von Neumann and Birkhoff. von Neumann established the mean ergodic theorem, proving convergence of ergodic averages in the Hilbert space norm, while Birkhoff obtained the corresponding pointwise convergence theorem. The mean ergodic theorem was subsequently extended by F. Riesz \cite{Riesz1938} to contractions on $L^p$-spaces. Since then, ergodic theorems have developed in many directions and remain an active area of research. In this article, we study ergodic averages associated with actions of amenable groups and prove pointwise ergodic theorems along Reiter sequences in both the classical and non-commutative settings.

\medskip
The study of classical ergodic theorems has crossed the boundaries of discrete, abelian group actions, and recent works involve the rich geometric theory of general amenable groups. In the classical setting, Calder\' on \cite{calderon1953general} first proved the pointwise ergodic theorem along a collection of subsets $(N_s)_s$ of $G$ satisfying doubling and asymptotic invariance conditions. Later Breuillard and Tessera \cite{breuillard2014geometry}, \cite{tessera2007volume} proved that the balls associated to any invariant metric quasi-isometric to word metric on the groups of polynomial growth satisfy the above two conditions; hence, they obtained the ergodic theorems for these groups. However, Lindenstrauss \cite{Lindenstrauss2001}, settled the question for general amenable groups for F\o lner sequence of sets satisfying the tempered condition, which is a generalization of doubling condition for these sets. A few more such classical ergodic theorems have already been obtained by various authors. We refer to \cite{nevo2006pointwise, anantharaman2010ergodic} for further details.

On the other hand, von Neumann algebras have emerged as non-commutative probability spaces due to its rich structure of projections and other heavy machinery analogous to the classical measure spaces. The study of ergodic theorems in the non-commutative spaces was pioneered by Lance \cite{lance1976ergodic}, in which he studied an analogue of the pointwise convergence of ergodic averages associated to any $*$-automorphism defined on any von Neumann algebra in the presence of an invariant faithful, normal state (not necessarily tracial). Immediately after, Lance's results were further extended for von Neumann algebras with more general maps than the $*$-automorphisms, such as various groups of $*$-automorphisms and in general for positive maps. We refer to \cite{conze1978ergodic}, \cite{kummerer1978non} and the references therein. It is known for a long time that the maximal inequality is essential to obtain the pointwise ergodic theorem \cite{wiener1939ergodic}. The first instance of maximal inequality was obtained by Yeadon \cite{Yeadon1977} for ergodic averages associated to the positive, sub-tracial maps on the non-commutative $L^1$-spaces. However, a more systematic study of maximal ergodic theorems is observed much later in the seminal work of Junge and Xu \cite{junge2007noncommutative}. Based on the ground breaking work of Pisier and Xu \cite{pisier2003non}, the authors formulated the non-commutative analogue of the norm of the maximal functions and showed that the aforementioned ergodic averages satisfy the maximal inequality and obtained associated pointwise ergodic theorems. After that, in the last few decades, the developments in the non-commutative Harmonic analysis have seen various applications of such non-commutative ergodic theorems. We refer to \cite{hong2024pointwise}, \cite{mei2007operator} and the references therein for more information.

The analogous results were then looked for in the non-commutative setting. But the non-commutative nature of the space of operators presents a hindrance in these studies. Firstly, for a collection of elements $(f_i)_i$ in the non-commutative $L_p$-spaces, the commutative way of defining maximal operators as $\sup_i \abs{f_i}$ is not feasible. Other than that the absence of analogues of covering lemma in this particular setting makes the study of maximal inequality in the non-commutative setting a much more subtle job. A maximal inequality and analogous pointwise ergodic theorems were first studied in \cite{hong_maximal_2023} for actions of groups of polynomial growth generalizing the classical results of Breuillard and Tessera \cite{breuillard2014geometry}, \cite{tessera2007volume}. There, the authors majorized the ergodic average associated to group action by a usual discrete average of a Markov map, hence obtaining the required maximal inequality using the same for Markov maps. Later on, this idea was again applied for the actions of unimodular locally compact amenable groups in \cite{chakraborty2025ergodic}. However, for general locally compact amenable groups, a different approach was taken by Cadilhac and Wang \cite{cadilhac2025noncommutativemaximalergodicinequalities}. They used the tiling property of the locally compact, amenable group to obtain a subsequence of a given F\o lner sequence along with an associated sequence of sets and partitions of the group, which resembles the geometric properties of the concentric circles and squares in the Euclidean spaces. Henceforth, they used the non-commutative square function estimates to deduce the weak type maximal inequality.

Nevertheless, the amenability of groups can be equivalently described by the Reiter property (see Definition \ref{Reiter-defn}). So, the existence of the Reiter sequence immediately raises the question of an analogous ergodic theorem along those Reiter sequences. As per our knowledge, we did not find any reference of such theorems even in the classical literature. Therefore, taking inspiration from \cite{cadilhac2025noncommutativemaximalergodicinequalities}, in this article we first ask the question that given a Reiter sequence, does there exist a suitable Reiter subsequence such that a non-commutative pointwise ergodic theorem can be obtained along that sequence and we answer it positively.

Our construction of such Reiter subsequence follows from the inherent geometric properties of the group itself. Eventually, we end up showing that every Reiter sequence has a further subsequence which is an infinite convex combination of normalized characteristic functions associated to certain F\o lner sequences and that the infinite convex combination is compatible with the non-commutative weak type maximal inequalities. In the following, we briefly describe our methodology. \Cref{main-thm} is the main geometric idea of this article.

\begin{enumerate}
    \item Given a Reiter sequence $(h_n)$ defined on the group, we first show that there exists a Reiter sequence consisting of simple functions satisfying the property 
    \begin{align*}
        \sum_{n \in \N} \norm{h_n-s_n}_1< \infty,
    \end{align*}
    and furthermore, we showed that with the help of this condition the ergodic averages associated to $(h_n)$ also satisfy the weak type inequality if the ergodic averages associated to $(s_n)$ have such a property.
    \item In the next step, to write the functions $h_n$ as an appropriate linear combination of characteristic functions of sets, the usual Layer-cake formula is employed. In the literature, this formula is essential in obtaining the F\o lner nets from the probability densities.
    \item Next, to bridge between the Reiter and F\o lner sequences, we observe that the $L_1$-translation error each $h_n$ turns out to be the average over its level sets of the F\o lner defects of those level sets and finally the average layer defect is found to be small. Furthermore, using Chebyshev's inequality, we observe that almost all layer-cake mass is concentrated on the ``good" levels.
	\item Just removing the bad levels, we get the required decomposition of $h_n$.
    \item Finally, with a rearrangement of the F\o lner sequence and the convex coefficients, we make every element of the simple Reiter sequence to be a suitable infinite sum of the normalized F\o lner sequence. See \Cref{decomposition_of_sequence}.
    \item Moreover, as observed in \cite{cadilhac2025noncommutativemaximalergodicinequalities}, every F\o lner sequence of a locally compact, second countable amenable group has a further subsequence, called admissible F\o lner sequence along which the ergodic averages converge pointwise. Analogously, we also construct a subsequence of the given Reiter sequence and show that maximal inequalities and the pointwise ergodic theorem remain true in this setting (see Lemma \ref{Adm_Folner_sequence}).
\end{enumerate}
The following is the main theorem of this article.
\begin{thm}
    Let $h=( h_n) $ be a Reiter sequence on $G$ and for a $\tau$-preserving dynamical $(\CM, G, \alpha)$, we consider the following averages 
$$ 
A_n^h(x)=:  \int_G \alpha_g(x) h_n(g)~ dm, ~ \text{ for } x \in \cup_{1\leq p < \infty} L^p(M, \tau) \text{ and } n \in \N.
$$
Then there exists a  subsequence $s=(s_n) $ of $(h_n)$ such that the sequence $(A_n^s)_n$ satisfies the weak type $(1,1)$ inequality and the strong type $(p,p)$ inequality for $p>1$. Furthermore, for every $x \in L^p(M, \tau)$, $A_n^s(x)$ converges in b.a.u. when $1 \leq p< \infty$ and in a.u. when $2 \leq p< \infty$.
\end{thm}
As an application of this theorem, we obtain a non-commutative Wiener-Wintner theorem along a Reiter sequence, which is a generalization of \cite[Theorem 1.1]{BGK_WW_2026}. In 1941, Wiener-Wintner theorem was first introduced by N. Wiener and A. Wintner in paper \cite{classicalZWW} for a measure preserving transformation $T$ acting on a measure space $X$ ergodically. Later, Bellow and Losert developed this theorem for different weight classes in \cite{BellowLosert}. A non-commutative counterpart of this theorem for the $\Z$-action on a von Neumann algebra $\CM$ with finite trace $\tau$ by an ergodic $\alpha \in Aut(\CM)$ was studied by S. Litvinov in \cite{Lit2014}, where he proved the uniform convergence of Wiener-Wintner averages of the form  
$$\frac{1}{N}\sum_{i=0}^N\lambda^i\alpha^i(x)\quad \lambda \in \T, x \in L^p(\CM,\tau)$$
in the sense of bilaterally almost uniform (b.a.u.) convergence. Similar kind of results were established for different weight classes in multiparameter case by Hong in \cite{Hong-Sun-2018} and M. O'Brien in \cite{OBrien2021NoncommutativeWT}. In 2014 Pavel Zorin-Kranich proved a Wiener-Wintner theorem for a continuous ergodic action of a locally compact second countable group $G$, preserving a probability measure $\mu$ on the compact metric space $X$ in \cite{Pavel}. In \cite{ElAbdalaoui2021}, Abdalaoui proved the Wiener-Wintner theorem using joining of two actions of $G$ on the measure space.  Our approach to proving the non-commutative Wiener-Wintner theorem along a Reiter sequence is similar to \cite{BGK_WW_2026}. Building on the main results of \cite{BGK_WW_2026}, we establish an appropriate version of the van der Corput lemma and use it to prove the desired result.

\subsection{Outline of the paper}
We organize the paper as follows. Section 2 contains basic notions regarding non-commutative $L^p$-spaces, vector-valued non-commutative $L^p$-spaces, and non-commutative dynamical systems. Section 3 is on Reiter sequences: we show that they can be approximated by simple functions. Using layer-cake decomposition, it can be expressed as infinite convex combinations of characteristic functions of admissible F\o lner sequences, after passing through a subsequence. In Section 4, we prove weak $(1,1)$ and strong $(p,p)$ maximal inequalities ($p > 1$) for the ergodic averages $A_n^h(x)$ along a Reiter sequence for a trace-preserving dynamical system $(\mathcal{M}, G, \alpha)$. This yields our main result (Theorem \ref{thm: bau and su convergence}) on the b.a.u. convergence of $A_n^h(x)$. In the last section, we present a non-commutative Wiener-Wintner theorem along a Reiter sequence as an application.

\section{Preliminaries}\label{preli}
Throughout this article, let \(\mathcal{M}\) denote a von Neumann algebra acting on a separable Hilbert space \(\mathcal{H}\), i.e., \(\mathcal{M}\subseteq \mathcal{B}(\mathcal{H})\). The operator norm inherited from \(\mathcal{B}(\mathcal{H})\) is written as \(\lVert \cdot \rVert_\infty\); when no confusion can arise, we abbreviate it simply as \(\lVert \cdot \rVert\). The commutant of \(\mathcal{M}\) is denoted by \(\mathcal{M}'\). We further write \(\mathcal{P}(\mathcal{M})\) for the collection of projections contained in \(\mathcal{M}\).

Let \(\mathcal{M}_*\) denote the predual of \(\mathcal{M}\). Elements of \(\mathcal{M}_*\) are called \emph{normal linear functionals} on \(\mathcal{M}\). A functional \(\phi \in \mathcal{M}_*\) is said to be \emph{positive} if
\[
\phi(x) \geq 0 \quad \text{for all } x \in \mathcal{M}_+,
\]
where \(\mathcal{M}_+\) denotes the cone of positive elements in \(\mathcal{M}\). A positive functional \(\phi \in \mathcal{M}_*\) is called a \emph{state} if \(\phi(1) = 1\). It is said to be \emph{faithful} if, for every \(x \in \mathcal{M}_+\),
\[
\phi(x) = 0 \quad \Longleftrightarrow \quad x = 0.
\]

For notational convenience, the term \emph{faithful normal} will be abbreviated as \emph{f.n.} Given an f.n.\ state \(\phi\) on \(\mathcal{M}\), we denote by \(L^2(\mathcal{M},\phi)\) the Hilbert space arising from the Gelfand--Naimark--Segal (GNS) construction associated with \(\phi\). For each \(a \in \mathcal{M}\), we write \(\widehat{a}\) for the corresponding vector in \(L^2(\mathcal{M},\phi)\). Moreover, we identify \(\mathcal{M}\) with its image in \(\mathcal{B}(L^2(\mathcal{M},\phi))\) under the canonical GNS representation.

\subsection{Non-commutative $L^p$-spaces  ($1\leq p\leq\infty $)}

Let $\mathcal{M} \subseteq \mathcal{B}(\mathcal{H})$ be a von Neumann algebra equipped with a faithful, normal, semifinite (f.n.s.) trace $\tau$. An (possibly unbounded) operator 
$$x: \mathcal{D}(x) \subseteq \mathcal{H} \to \mathcal{H},$$
which is densely defined, closed, and affiliated with $\mathcal{M}$ (i.e., satisfying $u'^* x u' = x$ for every unitary $u' \in \mathcal{M}'$), will be denoted as $x \,\eta\, \mathcal{M}$. Such an operator $x$ is said to be $\tau$-measurable if, for every $\varepsilon > 0$, there exists a projection $e \in \mathcal{P}(\mathcal{M})$ such that $\tau(1-e) < \varepsilon$ and $e\mathcal{H} \subseteq \mathcal{D}(x)$. The collection of all $\tau$-measurable operators affiliated with $\mathcal{M}$ is then defined by
\[
L^{0}(\mathcal{M},\tau) := \{\, x \,\eta\, \mathcal{M} : x \text{ is } \tau\text{-measurable}\,\}.
\]

The space $L^{0}(\mathcal{M},\tau)$ forms a $^*$-algebra under the adjoint operation, strong sum $x+y := \overline{x+y}$, and strong product $x \cdot y := \overline{xy}$, where $\overline{x}$ denotes the closure of $x$. Detailed descriptions can be found in \cite{stratila2019lectures, hiai2021lectures}.

Several natural topologies arise on $L^{0}(\mathcal{M},\tau)$. The first relevant one is the \emph{measure topology}, whose neighborhoods are of the form $x + \mathcal{N}(\varepsilon,\delta)$ for $\varepsilon,\delta > 0$ and $x \in L^{0}(\mathcal{M},\tau)$, where
\[
\mathcal{N}(\varepsilon,\delta) := \{\, x \in L^{0}(\mathcal{M},\tau) : \text{there exists } e \in \mathcal{P}(\mathcal{M}) \text{ with } \tau(1-e) < \delta \text{ and } \|exe\| < \varepsilon \,\}.
\]
A net $\{x_i\}_{i\in I}$ converges to $x$ in the measure topology precisely when, for every $\varepsilon,\delta>0$, there exists $i_0\in I$ such that for all $i\ge i_0$ one may find $e_i\in \mathcal{P}(\mathcal{M})$ with $\tau(1-e_i) < \delta$ and
\[
\|\, e_i (x_i - x)e_i \,\| < \varepsilon.
\]

By \cite[Theorem 1]{nelson1974notes}, the space $L^{0}(\mathcal{M},\tau)$ is complete in the measure topology. Furthermore, this topology is metrizable, and $\mathcal{M}$ is dense in $L^{0}(\mathcal{M},\tau)$.

The topologies of primary interest are the \emph{bilateral almost uniform} (b.a.u.) and \emph{almost uniform} (a.u.) topologies.

\begin{defn}\label{bau conv defn}
A net $\{x_i\}_{i\in I}\subseteq L^{0}(\mathcal{M},\tau)$ is said to converge \emph{bilaterally almost uniformly} (b.a.u.) (\textit{resp.} \emph{almost uniformly} (a.u.)) to $x\in L^{0}(\mathcal{M},\tau)$ if, for every $\varepsilon>0$, there exists a projection $e \in \mathcal{P}(\mathcal{M})$ such that $\tau(1-e) < \varepsilon$ and 
\[
\lim_{i} \|\, e(x_i - x)e \,\| = 0 \quad (\textit{resp. } \lim_{i} \|\, (x_i - x)e \,\| = 0).
\]
\end{defn}

Clearly, almost uniform (a.u.) convergence of a net implies bilateral almost uniform (b.a.u.) convergence. 

The trace $\tau$ on $\mathcal{M}_+$ extends in the standard way to $L^{0}(\mathcal{M},\tau)_+$ by
\[
\tau(x) := \int_0^{\infty} \lambda \, d\tau(e_\lambda),
\]
where $x = \int_{0}^{\infty} \lambda\, de_\lambda$ is the spectral decomposition of $x$. 

For $1\leq p <\infty$, consider the set 
\[
L^p(\mathcal{M}, \tau) := \{x\in L^0(\mathcal{M},\tau) : \tau(|x|^p)<\infty\}.
\]
This forms a normed linear space with the norm $\|x\|_p := \tau(|x|^p)^{1/p}$. For $p=\infty$, we define $L^\infty(\mathcal{M}, \tau) := \mathcal{M}$ with the norm $\|\cdot\|_\infty$ being the usual operator norm $\|\cdot\|$ in $\mathcal{M}$. The space $L^p(\mathcal{M}, \tau)$ is called the \emph{non-commutative $L^p$-space}. 

We recall the following theorem from \cite{hiai2021lectures}.

\begin{thm}\label{Density_thm}
For every $p \in [1,\infty]$,~~ $L^p(\mathcal{M}, \tau)$ is a Banach space with the norm $\|\cdot\|_p$. In particular, $L^2(\mathcal{M}, \tau)$ is a Hilbert space with the inner product
\[
\langle x,y\rangle = \tau(y^*x).
\]
Moreover, $\mathcal{M} \cap L^1(\mathcal{M}, \tau)$ is dense in $L^p(\mathcal{M}, \tau)$ for any $p \in [1,\infty)$.
\end{thm}

\subsection{Vector-valued non-commutative $L^p$-spaces}

Let $1 \leq p < \infty$. The space $L^p(\mathcal{M}; \ell^\infty)$ is defined as the set of all sequences $(x_n)_{n \geq 0}$ in $L^p(\mathcal{M})$ that admit a factorization of the form
\begin{equation}\label{facto}
    x_n = a y_n b \quad \text{for all } n \geq 0,  
\end{equation}
where $(y_n)_{n \geq 0}$ is a bounded sequence in $\mathcal{M}$, and $a, b \in L^{2p}(\mathcal{M})$. The norm on this space is defined by
\[ 
\|(x_n)_{n \geq 0}\|_{L^p(\mathcal{M}; \ell^\infty)} = \inf \left\{ \|a\|_{2p} \left( \sup_{n \geq 0} \|y_n\|_\infty \right) \|b\|_{2p} \right\},         
\]
where the infimum is taken over all possible factorizations of $(x_n)_{n \geq 0}$ satisfying \eqref{facto}.

\begin{prop}
    Let $(x_n)_{n \geq 0}$ be a sequence of positive elements in $L^p(\mathcal{M}, \tau)$. Then $(x_n)_{n \geq 0} \in L^p(\mathcal{M}; \ell^\infty)$ if and only if there exists an element $a \in L^p(\mathcal{M})_+$ such that 
    \[ 
    x_n \leq a \quad \text{for all } n \geq 0. 
    \]
    Furthermore, the norm is given by
    \[ 
    \|(x_n)_{n \geq 0}\|_{L^p(\mathcal{M}; \ell^\infty)} = \inf \big\{ \|a\|_p : a \in L^p(\mathcal{M})_+ \text{ such that } x_n \leq a \text{ for all } n \geq 0 \big\}.
    \]
\end{prop}
\begin{defn}[Strong Type $(p,p)$ Inequality]\label{strong_pp}
Let $1 \leq p < \infty$, and let $A = (A_n)_{n \in \mathbb{N}}$ be a sequence of maps on $L^p(\mathcal{M})$. We say that $A$ satisfies a \emph{strong type $(p,p)$ maximal inequality} (or simply the strong $(p,p)$ inequality) if $A$ defines a bounded linear operator from $L^p(\mathcal{M}, \tau)$ into the vector-valued space $L^p(\mathcal{M}; \ell^\infty)$. 
Explicitly, this means that there exists a constant $C_p > 0$ such that for every $x \in L^p(\mathcal{M})$, the sequence $(A_n(x))_{n \in \mathbb{N}}$ belongs to $L^p(\mathcal{M}; \ell^\infty)$ and satisfies
\[
\|(A_n(x))_{n \in \mathbb{N}}\|_{L^p(\mathcal{M}; \ell^\infty)} \leq C_p \|x\|_p.
\]

\medskip
\noindent
Equivalently,  for positive elements, 
by the definition of the $L^p(\mathcal{M}; \ell^\infty)$ norm, the sequence of positive maps $A = (A_n)_{n \in \mathbb{N}}$ satisfies the strong type $(p,p)$ inequality if and only if there exists a constant $C_p > 0$ such that for every positive element $x \in L^p(\mathcal{M})_+$, one can find a positive element $a \in L^p(\mathcal{M})_+$ satisfying:
\begin{enumerate}
    \item $A_n(x) \leq a \quad \text{for all } n \in \mathbb{N}$, and
    \item $\|a\|_p \leq C_p \|x\|_p$.
\end{enumerate}
\end{defn}

\medskip
\noindent
We now define the space \(\Lambda_{1, \infty}(\mathcal{M}; \ell^\infty)\), consisting of all sequences \(x = (x_n)_{n \in \mathbb{N}}\) in \(L^1(\mathcal{M})\) for which the quantity
\[
\|x\|_{\Lambda_{1, \infty}(\mathcal{M}; \ell^\infty)} := \sup_{\lambda>0} \lambda \inf \big\{ \tau(1-e) : e \in \mathcal{P}(\mathcal{M}), \ e x_n e \leq \lambda \text{ for all } n \in \mathbb{N} \big\}
\]
is finite. Equipped with this quantity as a quasi-norm, \(\Lambda_{1, \infty}(\mathcal{M}; \ell^\infty)\) is a \emph{quasi-Banach space}. 
\begin{defn}
Let  \(A = (A_n)_{n \in \mathbb{N}}\) be a sequence of maps on \(L^1(\mathcal{M})\). We say that \(A\) satisfies the \emph{weak type \((1,1)\) inequality} if \(A\) defines a bounded map from \(L^1(\mathcal{M})\) into \(\Lambda_{1, \infty}(\mathcal{M}; \ell^\infty)\).

\medskip
\noindent
Equivalently, 
for a sequence of maps  $A = (A_n)_{n \in \mathbb{N}}$  on $L^1(\mathcal{M}, \tau)$,  we say that $A$ satisfies the \emph{weak type $(1,1)$ inequality} if there exists a constant $C > 0$ such that for every positive element $x \in L^1(\mathcal{M})_+$ and every $\lambda > 0$, there exists a projection $e \in \mathcal{P}(\mathcal{M})$ satisfying:
\begin{enumerate}
    \item $\tau(1-e) \leq \frac{C}{\lambda} \|x\|_1$, and
    \item $e A_n(x) e \leq \lambda e \quad \text{for all } n \in \mathbb{N}$.
\end{enumerate}

\end{defn}

%

\subsection{Noncommutative dynamical systems}

\begin{defn}\label{nc dyn sys} 
Let $(E,\Vert{}\cdot\Vert{})$ be a real ordered Banach space, and let $G$ be a locally compact second countable (l.c.s.c.) group. A \emph{noncommutative dynamical system} is a triple $(E,G,\gamma)$, where $\gamma : G \to \mathcal{B}(E)$ is a mapping satisfying $\gamma_s \circ \gamma_t = \gamma_{st}$ for all $s,t \in G$, and the following conditions hold:
\begin{enumerate}
    \item For every $a \in E$, the orbit map $s \mapsto \gamma_s(a)$ is continuous on $G$, where $E$ is equipped with the norm topology (and the $w^*$-topology in the case where $E=\mathcal{M}$).
    \item The operators are uniformly bounded, i.e., $\displaystyle\sup_{s \in G}\|\gamma_s\| < \infty$.
    \item Positivity is preserved: if $a \ge 0$, then $\gamma_s(a) \ge 0$ for all $s\in G$.
\end{enumerate}
\end{defn}
We further assume that if $E = \mathcal{M}$, then $\gamma_{s}$ is a $^*$-automorphism on $\mathcal{M}$ for all $s \in G$. Furthermore, suppose $\mathcal{M}$ is a semifinite von Neumann algebra with a faithful normal weight $\tau$. We say that $(\mathcal{M}, G, \alpha)$ is $\tau$-preserving if $\tau \circ \alpha_g = \tau$ for all $g \in G$.

Let $(\mathcal{M}, G, \alpha)$ be a noncommutative dynamical system, where $G$ is a l.c.s.c. group acting on a semifinite von Neumann algebra $\mathcal{M}$ with a faithful normal tracial weight $\tau$. The action is given by a homomorphism $ \alpha : G \to \mathrm{Aut}(\mathcal{M}),$
where each $\alpha_g$ is a $^*$-automorphism of $\mathcal{M}$. This means that for every $g \in G$ and every $x \in \mathcal{M}$, $ \alpha_g(x^*) = \alpha_g(x)^*.$

\section{Averages Along Reiter Sequences}
In this section, we recall some basics of Reiter sequences on locally compact groups. 
Suppose  $G$ is a l.c.s.c.  group with a Haar measure $m$. We recall the following definitions. 
\begin{defn}\label{Reiter-defn}
A sequence $ (h_n)$    in $L^1(G, m)$ such that $$ h_n \geq 0, ~~ \int_G h_n d m = 1, ~ \text{ for all } n \in \N$$ is called a \emph{ left (resp. right) Reiter sequence } if 
$$ \lim_{ n \to \infty } \norm{ \lambda_g(h_n)- h_n}_1= 0, (\textit{resp.} \lim_{ n \to \infty } \norm{ \rho_g(h_n)- h_n}_1= 0)\text{ for all } g \in G,$$
where $ \lambda_g(h)(s) = h(g^{-1} s)$ and $\rho_g(h)(s) = h( sg)$ , for $ g, s \in G$ and $ h \in L^1(G, m)$. Equivalently, for any compact subset $K$ of $G$
\begin{align*}
    \lim_{n \to \infty} \sup_{g \in K}\norm{ \lambda_g(h_n)- h_n}_1= 0, ~~(\textit{resp.}\lim_{n \to \infty} \sup_{g \in K}\norm{ \rho_g(h_n)- h_n}_1= 0).
\end{align*}
The argument is similar to the F\o lner sequence case \cite{Paterson1988}.
\end{defn}

Now let $\CM$ be a semifinite von Neumann algebra with f.n.s trace $\tau$ and $ (\CM, G , \alpha) $ be a dynamical system preserving the f.n.s trace $ \tau$. Then for $g \in G$, the map $\alpha_g $ extends to $L^p(\CM, \tau)$, we denote the extension by same symbol $ \alpha_g$ and $( L^p(\CM, \tau), G, \alpha)$ becomes a non-commutative dynamical system.

\medskip
\noindent
Let $h=( h_n) $ be right Reiter sequence on $G$ with right Haar measure $m$ and for a $\tau$-preserving system $(\CM, G, \alpha)$, we consider the following averages 
$$ A_n^h(x)=:  \int_G \alpha_g(x) h_n(g)~ dm, ~ \text{ for } x \in L^p(\CM, \tau) \text{ and } n \in \N.$$
In this article our main aim is prove weak type $(1, 1)$ and strong $(p,p)$, $(p > 1)$ maximal inequality for averaging maps $A^h=( A_n^h)$. Indeed, we prove that there exists a subsequence $(n_k)$ such that the maps $(( A_{n_k}^h)$, satisfies the weak type $(1, 1)$ and strong $(p,p)$, $(p > 1)$ maximal inequality.   We begin by proving the following theorem as a first step. 
\begin{thm}\label{sum-maximal}
Let $G$ be a locally compact group with a right Haar measure $m$. Suppose $ (f_n)$ and $(h_n) $ are two right Reiter sequences such that  $ c=\displaystyle \sum_{ n=1}^\infty  \norm{ h_n-f_n}_{ L^1(G, m) } < \infty $ and suppose that $(A_n^f)$ satisfies the weak type $(1, 1)$ maximal inequality. Then $(A_n^h)$ also satisfies the weak type $(1, 1)$ maximal inequality.
\end{thm}

\begin{proof}
 First consider $ t = \underset{ n \in \N}{\sup} \abs{ h_n-f_n } $. Since $ \displaystyle \sum_{ n=1}^\infty  \norm{ h_n-f_n}_{ L^1(G, m) } < \infty $, we note that $ t \in L^1(G, m)$. Indeed, it follows that 
\[
t(x)=\sup_n |h_n(x)-f_n(x)|
\leq \sum_{n=1}^{\infty}|h_n(x)-f_n(x)|,
\]
Then by the monotone convergence theorem, we have  
\[
\|t\|_1
\leq \sum_{n=1}^{\infty}\|h_n-f_n\|_1<\infty.
\]
Hence, \(t\in L^1(G,m)\).

\medskip
\noindent
Let $ x \in L^1(\CM, \tau)_+$, then consider $ a = \int_G \alpha_g(x) t(g) dm$ and it follows that $ a \in L^1(M, \tau)_+$. Let $ w = (h_n-f_n)$, then   we note that 
$$ A^w_n(x) \leq a, ~~ \text{ for all } n \in \N .$$
Therefore, for given $ \lambda >0$, Consider the projection 
$e = \chi_{[0, \lambda]}(a)$, then note that 
$$ eA_n^w(x)e \leq \lambda e, ~ \text{ for all } n \in  \N \text{ and } \tau(1-e) \leq \frac{\norm{x}}{\lambda}.$$
Thus,  the maps $(A_n^w)  $ satisfies weak type $(1, 1)$ maximal inequality.
We note that $ A_n^h = A_n^w +A_n^f$ for all $n \in N$. As both the maps $(A_n^w )$ and  $(A_n^f)$ satisfies the weak type $(1, 1)$  maximal inequality, it follows that $ ( A_n^h)$
also  satisfies the weak type $(1, 1)$ maximal inequality.
\end{proof}

Now we like to establish that it is enough to take the Reiter sequence to the Reiter sequence of simple functions. 
Indeed, we have the following results. 
\begin{lem}\label{finite-sum-aprox}
 Let $G$ be a locally compact group with a right Haar measure $m$.  Suppose $(h_n)$ be a right Reiter sequence, and $(\eps_n)\subseteq (0, 1)$ is a sequence of positive reals with $ \displaystyle\lim_{n \to \infty} \eps_n = 0$; then there exists a right Reiter sequence $(s_n)$ consisting of simple functions such that 
 $$ \norm{ h_n-s_n}_{ L^1(G, m )} \leq \eps_n, \text{ for all } n \in \N. $$
\end{lem}
\begin{proof}
We  note that for each $ n \in \N$, there exists a sequence of  simple positive measurable function $(f_n)$ such that 
 $$ \norm{ h_n -f_n }_{L^1(G,m)} < \eps_n'$$
 where $ \eps_n' = \frac{\eps_n}{4}$.
Now let $ g \in G$, then we observe the following 
\begin{align*}
  \norm{ \rho_g(f_n) -f_n}_{ L^1(G, m )} &=  \norm{ \rho_g(f_n) -\rho_g(h_n) + \rho_g(h_n) - h_n + h_n -f_n}_{ L^1(G, m )} \\
  &= 2 \norm{ f_n -h_n}_{ L^1(G, m )}  + \norm{ \rho_g(h_n) -h_n}_{ L^1(G, m )} \\
  & \leq \eps_n'  + \norm{ \rho_g(h_n) -h_n}_{ L^1(G, m )} \xrightarrow{ n \to \infty } 0.
\end{align*}
Let $ a_n = \norm{ f_n}_{ L^1(G, m)}$. 
Then it follows that 
$$ a_n = \norm{ f_n}_{ L^1(G, m)} = \norm{ f_n-h_n}_{ L^1(G, m)}+ \norm{ h_n}_{ L^1(G, m)}\leq \eps_n' + 1.$$
On the other hand, we also have 
\begin{align*}
 &1= \norm{h_n}_{ L^1(G, m)} \leq \norm{h_n-f_n}_{ L^1(G, m)} + \norm{f_n}_{ L^1(G, m)} \leq a_n + \eps_n'\\
 \implies&  1-\eps_n' \leq a_n. 
\end{align*}
Thus, we have $ 1-\eps_n' \leq a_n \leq 1+\eps_n'.$ Hence,  $ a_n \to 1$ as $ n \to \infty.$ Set $ b_n = a_n^{ -1} \to 1 $ as $n \to \infty$. Now consider 
$$ s_n = b_n f_n.$$
Then note that $ (s_n) $ is Reiter sequence. Suppose there exists a $k \in \N$ such that $ \eps_n' \leq \frac{1}{k}$, for all $n \in \N$  and $ k > 1$ and observe that 
$ b_n \leq 2$. Consequently, we obtain the following
\begin{align*}
\norm{ s_n-h_n}_{ L^1(G, m )} &=\norm{ b_nf_n-h_n}_{ L^1(G, m )} \\
&=\norm{ b_nf_n -b_nh_n}_{ L^1(G, m )} + \abs{1-b_n}\norm{ h_n}_{ L^1(G, m)}\\
&\leq b_n \eps_n' + b_n \abs{ 1-a_n} \\
&\leq 2 b_n \eps_n' = 4\eps_n'= \eps_n. 
 \end{align*}
This completes the proof.

\end{proof}


\begin{rem}
     We note that Lemma \ref{finite-sum-aprox} holds for a left Reiter sequence under the left Haar measure. Using Theorem \ref{sum-maximal} and Lemma \ref{finite-sum-aprox}, we observe that it is enough to prove the weak type $(1, 1)$ maximal inequality for the maps $(A_n^h)$ associated to the Reiter sequence $(h_n)$ consisting of simple functions. 
\end{rem}

\medskip \medskip

\subsection{Approximation of Reiter Sequences by Convex Combinations of F\o lner Averages}
In this section, we establish that a Reiter sequence of simple functions can be approximated by a linear combination of a Reiter sequence of F\o lner sequences. Let $G$ be a l.c.s.c.  group with right Haar measure $m$ and $e$ be the identity element in $G$. We begin with a group-theoretic result related to the $U$-thickening of compact sets (see, e.g., \cite{colin-2023}). Although the technique is likely well known to specialists, we provide a full proof below to keep the presentation self-contained.

\begin{prop}\label{prop: int over good set}
	Let $K_0$ be a compact subset of $G$ containing $e$ and $U$ be a compact neighbourhood of $e$. Write $K=UK_0$. Then for any set measurable set $F \subseteq G$ with $ 0 <m(F) < \infty  $, define
	\begin{align*}
		\Phi :=\int_K  \frac{m(Fg\Delta F)}{m(F)}\,dm(g).
	\end{align*}
	If $\Phi\leq \eps$ and $\eps>0$, then 
	\begin{align*}
		\sup_{x\in K_0} \frac{m(Fx\Delta F)}{m(F)}\leq \frac{2 \eps}{m(U)}.
	\end{align*}
\end{prop}

\begin{proof}
	Fix $x\in K_0$. Then for every $g\in G$, we know
	\begin{align*}
		m(Fx\Delta F) \leq m(Fx\Delta Fg) + m(Fg\Delta F).
	\end{align*}
	Then integrating over $ Ux$, since $m(Ux)=m(U)$, we have
	\begin{align*}
		m(U)m(Fx\Delta F) \leq \int_{Ux} m(Fx\Delta Fg)\,dm(g) + \int_{Ux} m(Fg\Delta F)\,dm(g).
	\end{align*}
	To estimate the first integral, using the right invariance of Haar measure, we obtain
	\begin{align*}
		\int_{Ux} m(Fx\Delta Fg)\,dm(g) = \int_U m(F\Delta Fu)\,dm(u) \leq \int_{K} m(F\Delta Fg)\,dm(g).
	\end{align*}
	Here, we used $U\subseteq K$. For the second integral, since $Ux\subseteq K ,$ we have
	\begin{align*}
		\int_{Ux}m(Fg\Delta F)\,dm(g) \leq \int_{K}m(Fg\Delta F)\,dm(g).
	\end{align*}
	Now, combining the above inequalities, we obtain
	\[
	m(U)m(Fx\Delta F) \leq 2 \int_{K} m(Fg\Delta F)\,dm(g).
	\]
	Which implies
	\[
	\frac{m(Fx\Delta F)}{m(F)} \leq \frac{2}{m(U)} \int_{K} \frac{m(Fg\Delta F)}{m(F)} \,dm(g).
	\]
	The last integral is $\Phi$, so we have $\frac{m(Fx\Delta F)}{m(F)} \leq \frac{2\eps}{m(U)}.$ Since $x\in K_0$ was arbitrary, we have 
	\begin{align*}
		\sup_{x\in K_0} \frac{m(Fx\Delta F)}{m(F)}\leq \frac{2 \eps}{m(U)}.
	\end{align*}
\end{proof}
We present the following lemma for a right Reiter sequence but the result holds for a left Reiter sequence with a similar proof. Therefore, we present the details for the right case only.

\begin{lem}\label{lem: Reiter subsequence}
Let $(f_n)$ be a right Reiter sequence in $L^1(G, m)$, $(K_n)_{n=1}^\infty$ an increasing sequence of compact sets with $ \underset{n \in \N}{\cup} K_n= G$ , and $(\varepsilon_n)$ a sequence of positive reals.
Then there exists a subsequence $(h_k)$ of $ (f_n)$  such that
\[
\sup_{g \in K_k} \|\rho_g(h_k) - h_k\|_{ L^1(G, m)} < \varepsilon_k, ~~ \text{ for all } k \in \N.
\]
\end{lem}

\begin{proof}
By the definition of a Reiter sequence, for any compact subset $K \subseteq G$, the sequence of translation differences converges to zero uniformly on $K$. That is, for any $\varepsilon > 0$, there exists an $N_\varepsilon \in \mathbb{N}$ such that 
\[
\sup_{g \in K} \|\rho_g(f_n) - f_n\|_{ L^1(G, m)} < \varepsilon, ~\text{for all } n \ge N_\varepsilon.
\]
We are given an increasing sequence of compact subsets  $(K_n)$ of $G$ such that $ \cup K_n= G$. We construct the desired subsequence by induction.

\medskip
\noindent
For $k = 1$, corresponding to the compact set $K_1$ and  $\varepsilon_1 >0$, there exists an integer $N_1 \in \mathbb{N}$ such that
\[
\sup_{g \in K_1} \|\rho_g(f_n) - f_n\|_{ L^1(G, m)} < \varepsilon_1 \quad \text{for all } n \ge N_1.
\]
For $k = 2$, corresponding to the compact set $K_2$ and  $\varepsilon_2 > 0$, there exists an integer $N_2 > N_1$ such that
\[
\sup_{g \in K_2} \|\rho_g(f_n) - f_n\|_{ L^1(G, m)} < \varepsilon_2 \quad \text{for all } n \ge N_2.
\]
Proceeding inductively, for any $k \in \mathbb{N}$, we can find an integer $N_k \in \mathbb{N}$ strictly greater than $N_{k-1}$ such that
\[
\sup_{g \in K_k} \|\rho_g(f_n) - f_n\|_{ L^1(G, m)} < \varepsilon_k \quad \text{for all } n \ge N_k.
\]
Let us define a subsequence $(h_k)_{k=1}^\infty$ by setting $h_k = f_{N_k}$ for all $k \in \mathbb{N}$. Therefore, it follows directly from our construction that
\[
\sup_{g \in K_k} \|\rho_g(h_k) - h_k\|_{ L^1(G, m)} < \varepsilon_k \quad \text{for all } k \in \mathbb{N}.
\]
This completes the proof.
\end{proof}

Therefore, with the help of the preceding results, we may assume that 
  $(s_n)$ is a right Reiter sequence of simple functions in $L^1(G, m)$, after passing to a subsequence, we may  further assume that
\[
\sup_{g\in K_n}
\|\rho_g(s_n)-s_n\|_1
<
\varepsilon_n, ~~\text{ for all } g \in G \text{ and }
\sum_{n=1}^\infty \varepsilon_n < \infty .
\]
The following is inspired by the article I. Naomika (see \cite{Namioka-1964}) to prove the existence  F\o lner sequence, we use the layer cake decomposition tricks.

\begin{lem}\label{finite-simple}
Let $h \in L^1(G, m)$ be a positive simple function with $\|h\|_1=1$, and  $ A \in \CB_{ (0, \infty) }$. Then the integral representation over $A$ reduces to a finite sum:
\[
h=     \int_A \frac{1_{F_t}}{m(F_t)} \, dt + \int_{A^c} \frac{1_{F_t}}{m(F_t)} \, d\mu(t) = \sum_{j=1}^k a_j\frac{ 1_{F_j}}{m(F_j)} + \int_{A^c} \frac{1_{F_t}}{m(F_t)} \, d\mu(t),
\]
where $F_t := \{x \in G : h(x) > t\}$, $ E = \{ t>0 : m(F_t) > 0$\},$~\mu$ is a probability measure on $(0,\infty)$ and $\{a_j\}_{j=1}^k \subseteq [0, 1] $ with $ \displaystyle\sum_{ k =1}^n  a_k = \mu(A)$.
\end{lem}

\begin{proof}
As $F_t := \{x \in G : h(x) > t\}$, we note that 
$$ h = \int_0^\infty 1_{ F_t} ~d t.$$
Since  $h$ is a simple function, it assumes only a finite number of distinct positive values on $G$, which we order as $0 < c_1 < c_2 < \dots < c_k$. 
The level set $F_t$ only changes when $t$ crosses one of the values $c_j$. For $t$ between any two adjacent values, the set $F_t$ remains entirely constant. Consequently, the mapping $t \mapsto F_t$ produces only finitely many distinct sets, which we index as $F_1, F_2, \dots, F_k$.
We partition the set $A$ into equivalence classes where the level sets are constant. Define:
\[
    E_j := \{t \in A : F_t = F_j\}.
\]
Since every $t \in (0, \infty)$ corresponds to exactly one set $F_j$, the collection $\{E_j\}_{j=1}^k$ consists of mutually disjoint subsets that fully partition $A$. 
Suppose $ E = \{ t > 0 :  m(F_t) > 0 \}$.
Then using the partition, we decompose the integral over $A$ into a finite sum of integrals:
\begin{align*}
\int_A 1_{ F_t} dt &= \int_A \frac{1_{F_t}}{m(F_t)}  m(F_t) 1_E(t)dt\\
&=  \int_A \frac{1_{F_t}}{m(F_t)} \, d\mu(t), ~~ \text{ where } 1_E(t) m(F_t) dt = d\mu(t) \\
&= \sum_{j=1}^k \int_{E_j} \frac{1_{F_t}}{m(F_t)} \, d\mu(t).
\end{align*}
Now, using Fubini's theorem, we have
\begin{align*}
\mu(0,\infty)=\int_0^\infty m(F_t)dt=\int_0^\infty\int_G 1_{F_t}(x)dm(x)dt
=\int_Gh(x)dm(x)=1.
\end{align*}
For any fixed $j$, the level set $F_t = F_j$ is strictly constant for all $t \in E_j$. This allows us to factor the constant integrand out of each integral:
\begin{align*}
    \sum_{j=1}^k \int_{E_j} \frac{1_{F_t}}{m(F_t)} \, d\mu(t) &= \sum_{j=1}^k \frac{1_{F_j}}{m(F_j)} \int_{E_j} 1 \, d\mu(t) \\
    &= \sum_{j=1}^k \mu(E_j) \frac{1_{F_j}}{m(F_j)}.
\end{align*}
Take $ a_j = \mu(E_j)$, which satisfies the required conditions.
\end{proof}


We now prove our main theorem,  using Namioka's (see \cite{Namioka-1964}) layer-cake argument to extract approximately invariant sets from a sequence of $L^1$-functions satisfying Reiter's conditions. Indeed, by representing a non-negative function as an integral of its level sets and averaging, we isolate a family of sets with small relative translation error. Finally, we combine this with a compact-neighborhood thickening to ensure uniform invariance over the compact sets $K_n$.

In the following we assume that $G$ is a locally compact second countable, hence $\sigma$-compact group equipped with a right invariant Haar measure $m$ and we denote $$\mathcal P^1(G):=\left\{f\in L^1(G,m):f\geq 0,\ \norm{f}_1=1 \right\}.$$ Furthermore, let $(s_n)$ be a right Reiter sequence of simple functions in $L^1(G,m)$ and $(K_n)_n$ be an increasing sequence of compact subsets containing identity $e$ such that $m(K_n)>0,$ and $\displaystyle\cup_{n=1}^{\infty}K_n=G.$ Let $(\eps_n)$ be a positive sequence such that
\[
\eps_n\downarrow 0,
\qquad
\sum_{n=1}^\infty \sqrt{\eps_n}<\infty.
\]

The aim is to extract a subsequence $(h_n)$ of $(s_n)$ which is represented, up to a small $L^1$-error, as a finite convex combination of normalized
characteristic functions of sets which are almost invariant under $K_n$. 

The following is the main result of this section.

\begin{thm}\label{main-thm}
   There exists a subsequence $(h_n)$ of $(s_n)$ such that   for each $n \in \N$, there exists a  $k_n\in \N$,  positive real numbers $ \{ a_{n,j} > 0:   1\leq j \leq k_n \}  $,  a function $O_n$  and collections of measurable sets  $\{ F_{n, j} : 1\leq j\leq k_n \} $ satisfying the following:
 \begin{enumerate}
     \item $ \displaystyle \sum_{ j = 1}^{k_n} a_{n, j} =a_n \leq 1$ for all $n \in \N$ with $ a_n \xrightarrow{ n \rightarrow \infty } 1$ and $ \norm{ O_n}_{ L^1(G, m) } < \sqrt{\eps_n}, ~ $ for all $ n \in \N$.
     
\item  $ h_n = \displaystyle \sum_{j=1}^{k_n} a_{n,j} \frac{\mathbf{1}_{F_{n,j}}}{m(F_{n,j})} + O_n.$\\

\item $ \underset{ g \in K_n }{\sup }  \frac{ m(F_{n, j}g \Delta F_{n, j}) }{m(F_{n, j})} \lesssim \sqrt{\eps_n}, \text{ for all } n \in \N \text{ and } 1 \leq j\leq k_n.$

 \end{enumerate}
\end{thm}


     



\begin{proof}
 First, we fix a compact neighborhood $U$ containing the identity $e$ of $G$ such that $m(U)>0$. Then consider  the $U$-thickening of $K_n$ by
\[
\widetilde K_n:=UK_n.
\]
Since $K_n$ and $U$ are compact, $\widetilde K_n$ is compact, and $m( \widetilde K_n) >0$. 
By Lemma \ref{lem: Reiter subsequence}, we can find a subsequence $ (h_n)$ of $ (s_n)$ 
such that the translation error over these thickened sets satisfies the following; 
\begin{align}\label{error-control}
\sup_{g \in \widetilde{K}_n} \|\rho_{g^{-1}} (h_n) - h_n\|_{L^1(G,m)} \leq \frac{\eps_n}{m( \widetilde{K}_n)}.
\end{align}

\medskip
\noindent
For any $n \in \N$ and $t > 0$, consider the level sets
\[
    F_{n,t} := \{x \in G : h_n(x) > t\}.
\]
By Fubini's theorem, we have the following layer-cake formulation; 
$$
h_n(x) = \int_0^\infty 1_{F_{n,t}}(x) \, dt.
$$
Since $h_n$ is normalized i.e. $\int_G h_n \, dm = 1$ for all $n \in \mathbb{N}$, integrating over $G$ yields
\[
    \int_0^\infty m(F_{n,t}) \, dt = \int_G h_n(x) \, dm(x) = 1.
\]
We can therefore define a probability measure $\mu_n: \CB_{(0, \infty)} \to [0 ,1]$, where $ \CB_{(0, \infty)}$ is the Borel $\sigma$-algebra of $(0, \infty)$, by $d\mu_n(t) = 1_{ E_n}(t) m(F_{n,t}) \, dt$, where  $E_n = \{t > 0 : m(F_{n,t}) > 0\}$. Then  we can express $h_n$ as an integral convex combination of normalized characteristic functions:
\[
    h_n (x)= \int_0^\infty \frac{1_{F_{n,t}}(x)}{m(F_{n,t})} \, d\mu_n(t), \quad a.e.~~x\in G.
\]

Now, observing that $\rho_{g^{-1}}(h_n)(x)-h_n(x)=\int_0^\infty\big(1_{F_{n,t}g}(x)-1_{F_{n,t}}(x)\big)d\mu_n(t)$, using Fubini's theorem,
\begin{align*}
\norm{\rho_{g^{-1}}(h_n)-h_n}_1
&= \int_G \abs{\rho_{g^{-1}}(h_n)(x)-h_n(x)} \, dm(x) \\
&= \int_0^\infty \int_G \abs{1_{F_{n,t}g}(x)-1_{F_{n,t}}(x)} \, dm(x) \, dt \\
&= \int_0^\infty m(F_{n,t}g \mathbin{\Delta} F_{n,t}) \, dt \\
&= \int_0^\infty \frac{m(F_{n,t}g \mathbin{\Delta} F_{n,t})}{m(F_{n,t})} \, d\mu_n(t).
\end{align*}

Integrate this over $g \in \widetilde{K}_n$ with respect to the Haar measure $m$. By Fubini's Theorem, we can swap the order of integration between the Haar measure $m$ and the layer-cake probability measure $\mu_n$. Indeed, we obtain the following 
\begin{align*}
 \int_0^\infty \left( \int_{\widetilde{K}_n} \frac{m(F_{n,t}g \Delta F_{n,t})}{m(F_{n,t})} dm(g) \right) d\mu_n(t) 
 &=  \int_{\widetilde{K}_n}  \left(  \int_0^\infty \frac{m(F_{n,t}g \Delta F_{n,t})}{m(F_{n,t})} dm(g) \right) d\mu_n(t) \\
 &=    \int_{\widetilde{K}_n}  \left( \|\rho_{g^{-1}} (h_n) - h_n\|_{L^1(G,m)} \right)  dm(g)  \\
 & \leq  \frac{\eps_n}{m( \widetilde{K}_n)}  \cdot m( \widetilde{K}_n)= \eps_n    ~(\text{by Eq.} \ref{error-control}).
\end{align*}
We need to isolate the ``good'' levels. 
Let $\Phi_n(t)$ be the inner integral representing the average symmetric difference for a given level set $F_{n,t}$:
\begin{align}\label{phi-fun}
\Phi_n(t) = \int_{\widetilde{K}_n} \frac{m(F_{n,t}g \Delta F_{n,t})}{m(F_{n,t})} dm(g).
\end{align}
Choose a threshold $\delta_n = \sqrt{ \eps_n}$. Now we consider the ``bad'' levels where the average error is too large:
$$B_n = \{t > 0 : \Phi_n(t) > \delta_n\}.$$
By Chebyshev's inequality,  the measure of the bad levels is asymptotically vanishing and good levels converge to $1$; indeed, we note that 
\begin{align}\label{bad-set-control}
\mu_n(B_n) \leq \frac{\eps_n}{\delta_n} = \sqrt{ \eps_n} \xrightarrow{n \to \infty} 0~ ~\text{ and }  ~~\mu_n(B_n^c) \xrightarrow{n \to \infty} 1.
\end{align}

\noindent Then we  apply Proposition \ref{prop: int over good set} to obtain the following;
\begin{align}\label{Folner-proved}
\sup_{x \in K_n} \frac{m(F_{n,t}x \Delta F_{n,t})}{m(F_{n,t})} \leq \frac{2 \delta_n}{m(U)} \lesssim \sqrt{\eps_n}, ~~\text{ for all } t \in B_n^c.
\end{align}
This implies that the level sets $F_{n,t}$ for $t \in B_n^c$ are asymptotically $K_n$-invariant, meaning they form a sequence of F\o lner sets as $n \to \infty$.

\medskip
Finally, we decompose the integral representation of $h_n$ over $B_n^c$ and $B_n$:
\[
    h_n = \int_{B_n^c} \frac{1_{F_{n,t}}}{m(F_{n,t})} \, d\mu_n(t) + \int_{B_n} \frac{1_{F_{n,t}}}{m(F_{n,t})} \, d\mu_n(t).
\]
Suppose $ O_n(g)  =  \int_{B_n} \frac{1_{F_{n,t} } (g)}{m(F_{n,t})} \, d\mu_n(t)$. Note that $ O_n \in L^1(G, m ) $, indeed, we observe that 
$ \norm{O_n}_{ L^1(G, m)}  \leq \mu_n(B_n)  \leq \sqrt{ \eps_n} $ (see Eq.\ref{bad-set-control}).
Moreover, using the Lemma \ref{finite-simple}, we can find a $k_n $ and a set of positive elements $ \{ a_{n, j} : ~ 1 \leq j \leq k_n \} $ such that 
\[
    h_n =  \sum_{ j =1}^{k_n} a_{n, j}  \frac{1_{F_{n,j}}}{m(F_{n,j})}  + O_n.
\]
and $\displaystyle \sum_{j = 1}^{k_n} a_{ n, j}= \mu_n(B_n^c)  \leq 1.$ Furthermore, let $ a_n = \displaystyle \sum_{j = 1}^{k_n} a_{ n, j} $ and it follows that $  a_n \leq 1 $ and $ a_n \xrightarrow{ n \rightarrow \infty } 1$. This completes the proof. 
\end{proof}

\begin{rem}\label{rem: Folner cond}

    Since the group $G$ is assumed to be locally compact second countable, it is $\sigma$-compact and hence it is known that if for some $t\in B_n^c$, and $\eps>0$, there is $n \in \N$, such that $\frac{m(F_{n,t}g \Delta F_{n,t})}{m(F_{n,t})} \leq \eps$, for all $g \in K$, ($K \subseteq G$ compact), then it is equivalent to the condition that for that particular $t>0$, and for any $\eps>0$ and any compact set $K \subset G$, there is $n \in \N$ such that $\frac{m(F_{n,t}K \Delta F_{n,t})}{m(F_{n,t})} \leq \eps$, which will now imply the definition of F\o lner set considered in \cite{cadilhac2025noncommutativemaximalergodicinequalities}.

\end{rem}

\begin{rem}
We like to highlight that when $G$ is a countable discrete group, then the proof of Theorem \ref{main-thm} is a little less involved. Note that since discrete groups are unimodular, the left and right Haar measures coincide. 
Indeed, establishing the following 
$$
\sup_{g \in K_n} \frac{m(F_{n,t}g \mathbin{\Delta} F_{n,t})}{m(F_{n,t})} \lesssim \sqrt{ \eps_n},
$$
is easier than the continuous case.  But other techniques are similar. Indeed, 
let
 $G$ be a countable discrete group with the counting measure $m$. Suppose 
$(K_n)$ is a increasing sequence of finite  sets in $G$ with  $ \cup_{ n \in \N} K_n = G$.
Then it is enough  to find a subsequence $ (h_n)$ of $ (s_n)$ 
such that the translation errors over the finite sets $K_n$ satisfy the following; 
\begin{align}\label{error-control-1}
\sup_{g \in {K}_n} \|\rho_{g^{-1}} (h_n) - h_n\|_{L^1(G,m)} \leq \frac{\eps_n}{m( {K}_n)}.
\end{align}
Then we take the function $\Phi_n $ (see the Eq.\ref{phi-fun}) as the following 
\[
    \Phi_n(t) := \sup_{g \in K_n} \frac{m(F_{n,t}g \mathbin{\Delta} F_{n,t})}{m(F_{n,t})}.
\]
Now we observe the following; 
\begin{align*}
  \int_0^\infty \Phi_n(t) d\mu_n &= \int_0^\infty \sup_{g \in K_n} \frac{m(F_{n,t}g \mathbin{\Delta} F_{n,t})}{m(F_{n,t})} d\mu_n \\
  &\leq   \sum_{ g \in K_n} \int_0^\infty  \frac{m(F_{n,t}g \mathbin{\Delta} F_{n,t})}{m(F_{n,t})} d\mu_n \\
  &=  \sum_{ g \in K_n}\|\rho_{g^{-1}} (h_n) - h_n\|_{L^1(G,m)}   \leq  m( K_n) \frac{\varepsilon_n}{m(K_n)}= \eps_n    ~ (\text{see Eq.} \ref{error-control-1}).
\end{align*}
Thus, this establishes the fact that  $\int_0^\infty \Phi_n(t) \, d\mu_n(t) < \varepsilon_n$.
The rest is similar to what we did in Theorem \ref{main-thm}. Indeed, 
take $\delta_n = \sqrt{\varepsilon_n}$) as in . Define the set of ``bad'' levels by
\[
    B_n := \{t > 0 : \Phi_n(t) > \delta_n\}.
\]
Then deduce that the measure of the bad set vanishes asymptotically:
\[
    \mu_n(B_n) \leq \frac{1}{\delta_n} \int_0^\infty \Phi_n(t) \, d\mu_n(t) \leq \frac{\varepsilon_n}{\delta_n} \to 0.
\]
For the ``good'' levels $B_n^c$, we have $\mu_n(B_n^c) \to 1$. Crucially, for every $t \in B_n^c$, we are guaranteed that
\begin{align*}
\sup_{g \in K_n} \frac{m(F_{n,t}g \mathbin{\Delta} F_{n,t})}{m(F_{n,t})} \leq \delta_n.
\end{align*}
To overcome this difficulty in the continuous setting, we employ the technique of $U$-thickening the given compact sets $(K_n)$ and then proceed with the subsequent analysis (see Proposition \ref{prop: int over good set}).

 Although this represents a necessary topological innovation relative to Namioka's original layer-cake argument \cite{Namioka-1964}, which was inherently restricted to discrete groups and finite sets, the thickening trick is deeply rooted in the study of continuous groups (see \cite{Greenleaf1969} \cite{Emerson1967}). 
 This averaging mechanism has long served as the standard bridge between discrete and continuous amenability. By thickening the target sets with a compact neighbourhood of positive Haar measure, we are able to smooth the $L^1$-data, thereby bypassing the measure-theoretic anomalies of uncountably infinite spaces to achieve uniform F\o lner-type invariance.

\end{rem}

\section{Maximal inequality for the averages along Reiter sequence  }
Recall that  $ \CM$ is a von Neumann algebra with a f.n.s trace $\tau$ and $G$ is a l.c.s.c group with right Haar measure $m$. Let $(\CM, G, \alpha) $ be a  $\tau$-preserving non-commutative dynamical system.   Let $ h= (h_n)$ be a Reiter sequence in $L^1(G, m)$. Recall the following averages 
$$ A_n^h(x) = \int_G \alpha_g(x) h_n(g) dm, ~~\text{ for } x \in L^p(M, \tau). $$
We like to prove a weak type $(1, 1)$ maximal inequality for an appropriate subsequence of $(A_n^h) $.
In view of Theorem \ref{sum-maximal} and   Theorem  \ref{main-thm}, it is enough to prove a weak type $(1, 1)$ maximal inequality for a subsequence of  $(A_n^s) $ associated to a Reiter sequence $s=(s_n)  $ with the following form;
\begin{align}
    s_n = \displaystyle \sum_{j=1}^{k_n} a_{n,j} \frac{\mathbf{1}_{F_{n,j}}}{m(F_{n,j})} \text{ and }  \underset{ g \in K_n }{\sup }  \frac{ m(F_{n, j}g \Delta F_{n, j}) }{m(F_{n, j})} \lesssim \sqrt{\eps_n}, ~\text{ for all }  1 \leq j\leq k_n,
\end{align} 
where $ (K_n)$ is a given sequence of compact subsets of $G$ containing the identity with $ \cup_{n\in \N}K_n = G$, $(\eps_n) \subseteq (0, 1)$ with $ \displaystyle\sum_n \sqrt{ \eps_n} < \infty$ and $ k_n \in \N$ for all $n \in \N$ and $ \{ a_{n, j} :   1 \leq j \leq k_n \} \subseteq [0, 1]$ with  $ \displaystyle \sum_{ j = 1}^{k_n} a_{n, j}  \leq 1$.

\medskip
\noindent
Now for each $ n \in \N $ and for $ 1 \leq j \leq k_n $, we  define the average ${A}_{n,j}^s=\frac{1}{m(F_{n,j})}\displaystyle\int_{F_{n,j}}\alpha_g(x)dm(g).$   Therefore,   we can write the average 
$$
A_n^s(x)=\sum_{j=1}^{k_n} a_{n,j} {A}^s_{n,j}(x),  \text{ for each } n\in \N.
$$
We seek to extend the finite sequence $(a_{n,j})_{j=1}^{k_n}$ and its associated F\o lner sets into infinite sequences over $j$, while ensuring that they decrease with respect to $n$ and preserve the sum $A^s_{n}(x)$.
 Since removing the average ${A}^s_{n,j}(x)$ whose corresponding coefficient $a_{n,j}$ is zero does not affect the original average $A^s_n(x)$, we may assume $a_{n,j}> 0$ for all $j,n \in \N.$ Now we will proceed as follows:


\begin{lem}\label{decomposition_of_sequence}
    Let $(s_n)$ be a sequence of non-negative simple functions defined by 
    \[
    s_n = \sum_{j=1}^{k_n} a_{n,j} \frac{\mathbf{1}_{F_{n,j}}}{m(F_{n,j})}. 
    \]
   Suppose $ \sum_{j=1}^{k_n}a_{n,j}=a_n$, then $s_n$ can be rewritten as an infinite sum 
    \[
    s_n = \sum_{j=1}^{\infty} a'_{n,j} \frac{\mathbf{1}_{F_{n,j}'}}{m(F_{n,j}')},
    \]
    such that the following properties hold:
    \begin{enumerate}
    \item $ F_{n, j}'\in  \{ F_{n, 1 }, F_{n, 2 }, \cdots, F_{n, k_n }\} $ for all $j \in \N$.
        \item $a'_{n,j} > 0$ for all $j \ge 1$, $\displaystyle\sum_{j=1}^{\infty}a'_{n,j}=a_n$ and $ \displaystyle \sum_{ j=1}^\infty  \sqrt{ a_{n, j}'} < \infty  $.
        \item For each fixed $j$, the sequence $(a'_{n,j})_{n=1}^{\infty}$ decreases in $n$.
    \end{enumerate}
\end{lem}

\begin{proof}
    Consider the first row $(n=1)$, which is initially given as $(a_{1,j})_{j=1}^{k_1}$. We define the new infinite sequence of coefficients as follows:
    \[
    a'_{1,j} := 
    \begin{cases} 
        a_{1,j} & \text{for } 1 \leq j \leq k_1 - 1, \\
        \frac{a_{1,k_1}}{2^{j-k_1+1}} & \text{for } j \geq k_1.
    \end{cases}
    \]
    The corresponding new F\o lner sequence sets are defined as:
    \[
    F'_{1,j} := 
    \begin{cases} 
        F_{1,j} & \text{for } 1 \leq j \leq k_1, \\
        F_{1, k_1} & \text{for } j \geq k_1 + 1.
    \end{cases}
    \]
Next, we construct the second row $(n=2)$. If $a_{2,j} \leq a'_{1,j}$ for all $j$, we define the new second row exactly as we did for the first row. Otherwise, we proceed by redistributing the weights. Let $1 \leq i \leq k_2$ be the first index such that $a_{2,j} \leq a'_{1,j}$ for $1 \leq j \leq i$, but $a_{2,i+1} > a'_{1,i+1}$. We begin modifying the sequence by setting:
    \[
    a''_{2,j} := 
    \begin{cases} 
        a_{2,j} & \text{for } 1 \leq j \leq i, \\
        a'_{1,i+1} & \text{for } j = i+1.
    \end{cases}
    \]
    For the next term, we evaluate the remainder $a_{2,i+1} - a'_{1,i+1}$. If this remainder is less than or equal to $a'_{1,i+2}$, we set:
    \[
    a''_{2,i+2} := a_{2,i+1} - a'_{1,i+1}. 
    \]
    Otherwise, we cap it by setting $a''_{2,i+2} := a'_{1,i+2}$, and consider the new remainder $a_{2,i+1} - a'_{1,i+1} - a'_{1,i+2}$. We repeat this shifting process until the remaining part becomes strictly smaller than the corresponding $a'_{1,j}$. This process clearly ends in finitely many steps. 
    
\medskip

    Suppose that the process for $a_{2,i+1}$ terminates at the index $i'$. We then move to the next term from the second row, $a_{2,i+2}$, and continue the same redistribution process starting against $a'_{1, i'+1}$. Since there are only finitely many terms in the original second row, this entire redistribution concludes after a finite number of steps. Note that the terms generated at each step are strictly positive. 

    \medskip
    \noindent
    Let the final non-zero term obtained in this process be $a''_{2,j'}$ at some index $j'$. If $j' < k_1$, we replace $a''_{2,j'}$ by $\frac{a''_{2,j'}}{2}$ and continue the infinite halving process. Thus, without loss of generality, we may assume $j' \geq k_1$. We can now formally define the infinite sequence for the second row as:
    \[
    a'_{2,j} := 
    \begin{cases} 
        a''_{2,j} & \text{if } j < j', \\
        \frac{a''_{2,j'}}{2^{j-j'+1}} & \text{if } j \geq j'.
    \end{cases}
    \]
    By construction, this ensures that $0 < a'_{2,j} \leq a'_{1,j}$ for all $j$, and that $\displaystyle\sum_{j=1}^{\infty} a'_{2,j}=a_2$. Moreover, define $ F_{n, j}'$ as the first row in the tail part, and in between we keep the same $F_{2, j} $ with the decompositions of $ a_{2, j}$ to obtain the following: 
$$  \sum_{j=1}^{\infty} a'_{n,j} \frac{\mathbf{1}_{F_{n,j}'}}{m(F_{n,j}',)} =  \sum_{j=1}^{k_n} a_{n,j} \frac{\mathbf{1}_{F_{n,j}}}{m(F_{n,j})} $$

    \medskip
    For the general inductive step, given the $n$-th row $(a'_{n,j})_{j=1}^{\infty}$, we can construct the $(n+1)$-th row $(a'_{n+1,j})_{j=1}^{\infty}$ using the exact same redistribution argument. This guarantees that $0 < a'_{n+1,j} \leq a'_{n,j}$ and $\displaystyle\sum_{j=1}^{\infty} a'_{n+1,j}=a_{n+1}$, which completes the proof. 
\end{proof}


\medskip
In order to prove the weak type $(1, 1)$ and the strong $(p,p)$ maximal inequality, we need the F\o lner sequences to be admissible in the sense given in the paper \cite{cadilhac2022noncommutative}. But from the construction of $(F_{n,j})$, it is not guaranteed. However, from Remark \ref{rem: Folner cond}, it is clear that $(F_{n,j})_n$ is a sequence of F\o lner sets. Hence, we want to find an admissible subsequence of $(F_{n,j})$. Now we discuss the construction of the admissible F\o lner subsequence. 

Thus, we discuss the following. First, recall the definition of an admissible F\o lner sequence in a locally compact, $\sigma$-compact group $G$. We begin with the following definitions. 
\begin{defn}
    Let $K$ be a compact subset of the group $G$ and $E \subset G$, the $K$-\emph{boundary} of the set $E$ is defined as the following
    \begin{align*}
        \partial_K(E):= \cup_{g \in G}\{ gK : gK \cap E \neq \emptyset, \text{ and } gK \cap E^c \neq \emptyset \}.
    \end{align*}
    Now for $\eps>0$, the set $E$ is said to be $(\eps, K)$-\emph{boundary invariant} if 
    \begin{align*}
        \abs{\partial_K(E)} < \eps \abs{E}.
    \end{align*}
\end{defn}

\begin{defn}
    Let $E$ be a subset of the group $G$, the \emph{lower Banach density} of the set is defined by 
    \begin{align*}
        \underline{B}(E):= \sup_{K \subseteq G} \inf_{g \in G} \frac{\abs{gK \cap E}}{\abs{gK}},
    \end{align*}
    where the supremum is taken over all compact subsets $K$ of the group $G$.
\end{defn}



\begin{defn}
    A \emph{Filtered F\o lner sequence} is a triple $\CS:=((F_n)_n, (B_n)_n, (\CP_n)_n)$, where $(F_n)_n$ is a F\o lner sequence of subsets of $G$, $(B_n)_n$ is a sequence of compact subsets of $G$ and $(\CP_n)_n$ is a sequence of partitions of the group $G$, satisfying the following conditions:
    \begin{enumerate}
        \item for any $n \in \N$, and $A \in \CP_n$, there exists $\gamma \in G$ such that $A \subseteq \gamma B_n$,
        \item for any $n \geq k$, the set $F_n$ is $(2^{n-k}, B_k)$- boundary invariant.
    \end{enumerate}
    Now an atom $A \in \CP_k$ is said to be \emph{admissible} if for all $n < k$, the set $A$ is $(2^{n-k}, F_n)$-invariant (that is, $\abs{AF_n \setminus A} \leq 2^{n-k} \abs{F_n}$ and $\abs{AF_n^{-1} \setminus A} \leq 2^{n-k} \abs{F_n}$). Let $\CP_n^a$ denotes the collection of admissible atoms of $\CP_n$ and define
    $$Adm(\CS):= \cap_{n\geq 0} \cup_{A \in \CP_n^a} A.$$
    \noindent Let $0<c<1$, we say that $\CS$ is $c$-\emph{admissible} if $\underline{B}(Adm(\CS)) \geq 1-c$. By a slight abuse of notation, we say that a F\o lner sequence of sets $(F_n)_n$ of $G$ is called \emph{admissible} if there exists a triple $\CS:=((F_n)_n, (B_n)_n, (\CP_n)_n)$ and $c \in (0,1)$ satisfying the above conditions.
\end{defn}

\begin{rem}\label{adm_subsequence}
    Let $\CS:=((F_n)_n, (B_n)_n, (\CP_n)_n)$ be a $c$-admissible F\o lner sequence for some $c \in (0,1)$. Consider $\CS':=((F_{n_k})_k, (B_{n_k})_k, (\CP_{n_k})_k)$. Then observe that $Adm(\CS) \subseteq Adm(\CS')$. Therefore, for all compact set $K$ of $G$ and $g \in G$,
    \begin{align*}
        \frac{\abs{gK \cap Adm(\CS)}}{\abs{gK}} \leq \frac{\abs{gK \cap Adm(\CS')}}{\abs{gK}}.
    \end{align*}
    Hence, it is clear that $1-c \leq \underline{B}(Adm(\CS)) \leq \underline{B}(Adm(\CS'))$. So, $\CS'$ is also a $c$-admissible F\o lner sequence. We conclude that a subsequence of a $c$-admissible F\o lner sequence is $c$-admissible.
\end{rem}

\begin{lem}\label{Adm_Folner_sequence}
    Let $(s_n)$ be a sequence of simple functions of the form $$s_n = \sum_{j=1}^{\infty} a_{n,j} \frac{\mathbf{1}_{F_{n,j}}}{m(F_{n,j})}.$$
where $(F_{n,j})$ is a F\o lner sequence for each $j$ and there exists $k_n \in \N$ such that $F_{n,j}=F_{n,k_n}$ for all $j \geq k_n.$ Then there exists a subsequence $(s_{n_t})$ such that $(F_{n_t,j})_{ t \in \N}$ is an admissible F\o lner sequence. 
\end{lem}

\begin{proof}
    We will construct the subsequence using a standard diagonal argument. 
    For $j\in \{1, \cdots, k_1\}$, first consider the F\o lner sequences $(F_{n,j})_{n=1}^{\infty}$. By Theorem A of \cite{cadilhac2022noncommutative} and Remark \ref{adm_subsequence}, we can find a subsequence of indices $N_1 \subseteq \mathbb{N}$ such that $(F_{n,j})_{n \in N_1}$ is an admissible F\o lner sequence for all  $j\in \{1, \cdots, k_1\}$.
    Let $n_1$ be the least element of $N_1$ and consider
    \begin{align*}
        s_{n_1}:= \sum_{j=1}^\infty a_{n_1,j} \frac{\mathbf{1}_{F_{n_1,j}}}{m(F_{n_1,j})}.
    \end{align*}
    Next, noting that $(F_{n,j})_{n \in N_1}$ are F\o lner sequences, where $j\in \{1, \cdots, k_{n_1}\}$, we again deploy Theorem A of \cite{cadilhac2022noncommutative} and Remark \ref{adm_subsequence}, to extract admissible F\o lner subsequences $(F_{n,j})_{n \in N_2}$, where $N_2 \subseteq N_1$ and $j\in \{1, \cdots, k_{n_1}\}$. Let $n_2$ be the least element of $N_2$ and consider
    \begin{align*}
        s_{n_2}:= \sum_{j=1}^\infty a_{n_2,j} \frac{\mathbf{1}_{F_{n_2,j}}}{m(F_{n_2,j})}.
    \end{align*}
    In general, for $l \geq 2$, let $(F_{n,j})_{n \in N_{l-1}}$,  where $j\in \{1, \cdots, k_{n_{l-1}}\}$ are F\o lner sequences. Applying Theorem A of \cite{cadilhac2022noncommutative} and Remark \ref{adm_subsequence}, we get admissible F\o lner subsequences $(F_{n,j})_{n \in N_{l}}$ for $j\in \{1, \cdots, k_{n_{l-1}}\}$, where $N_l \subseteq N_{l-1}$. Let $n_l$ be the least element of $N_l$ and write
    \begin{align*}
        s_{n_l}:= \sum_{j=1}^\infty a_{n_l,j} \frac{\mathbf{1}_{F_{n_l,j}}}{m(F_{n_l,j})}.
    \end{align*}
    Observe that for every $j \in \N$, $(F_{n_l,j})_{l \in \N}$ is an admissible F\o lner sequence. This completes the proof.
    \end{proof}

\begin{rem}
By the Lemma \ref{decomposition_of_sequence} we may assume that $${A}_n^s(x)= \sum_{j=1}^{\infty}a_{n,j}{A}_{n,j}^s(x)$$ where $\sum_{j=1}^{\infty}a_{n,j}=a_n$,  $a_n \to 1$, $\displaystyle \sum_{j=1}^{\infty}\sqrt{a_{n,j}} < \infty $ and for each $j$, $a_{n,j}$ is decreasing (i.e. $a_{n+1,j} \leq a_{n,j}$).
    In view of Lemma \ref{Adm_Folner_sequence} passing through a subsequence, we may assume that $(F_{n,j})_n$ is an admissible F\o lner sequence for all $j \in \N.$ 
\end{rem}
  Next theorem proves the weak type $(1,1)$ maximal inequality. 
\begin{thm}\label{weak_type_(1,1)_ergodic}
    Let $x \in L^1(\CM)$ and $\lambda > 0$. Then there exists a projection $e_x \in \CM$ such that 
    \begin{align*}
        e_x{A}_n^s(x)e_x \leq \lambda~~\text{for all }n \in \N ~~\text{and}~~\tau(1-e_x)\leq C\frac{\|x\|_1}{\lambda}.
    \end{align*}
    For $p > 1$, $ ( A_n^s)$ satisfies a strong $(p,p)$ maximal inequality, indeed, for $ x \in L^p(\CM, \tau) $ we have 
    $$ \norm{\left( {A}^s (x)\right)}_{ L^p(\CM; \ell^\infty)} \lesssim \norm{x}_p. $$
\end{thm}
\begin{proof}
    Since ${A}_{n,j}^s(x)=\frac{1}{m(F_{n,j})}\int_{F_{n,j}}\alpha_g(x)dm(g)$ and for all $j,$ $(F_{n,j})$ is an admissible F\o lner sequence, by \cite{cadilhac2022noncommutative} ${A}_{n,j}^s(x)$ satisfies weak type $(1,1)$ maximal inequality. Hence, there exist $C'>0$ (independent of the F\o lner sequences) and $e_{j,x} \in \CM$ with
    
    $$e_{j,x}{A}_{n,j}^s(x)e_{j,x} \leq \frac{\lambda}{\sqrt{a_{1,j}}}\cdot \frac{1}{\big(\displaystyle\sum_{j=1}^\infty \sqrt{a_{1,j}}\big)} \text{ and } \tau(1-e_{j,x}) \leq C'\frac{\|x\|_1\cdot \sqrt{a_{1,j}}}{\lambda}\sum_{j=1}^\infty \sqrt{a_{1,j}}.$$
    
    Let us define $e_x:= \displaystyle\bigwedge_je_{j,x}.$
    Hence, $$\tau(1-e_x) \leq \sum_{j=1}^{\infty}\tau(1-e_{j,x}) \leq C'\frac{\|x\|_1}{\lambda}\big(\sum_{j=1}^{\infty} \sqrt{a_{1,j}}\big)^2 \leq C\frac{\|x\|_1}{\lambda} $$
where $C=C'\cdot \big(\sum_{j=1}^{\infty} \sqrt{a_{1,j}}\big)^2.$ Now from the Lemma \ref{decomposition_of_sequence} we note that $a_{n,j} \leq a_{1,j}$ for all $j\in \N.$ Therefore, we have 
\begin{align*}
    e_x{A}_{n}(x)e_x&=\sum_{j=1}^{\infty}a_{n,j}e_x{A}_{n,j}(x)e_x \\
    &\leq \sum_{j=1}^{\infty}a_{n,j}e_{j,x}{A}_{n,j}(x)e_{j,x}\\
    &\leq \sum_{j=1}^{\infty}a_{1,j}e_{j,x}{A}_{n,j}(x)e_{j,x}\\
    &\leq \sum_{j=1}^{\infty}a_{1,j} \big(\frac{\lambda}{\sqrt{a_{1,j}}}\big)\cdot \frac{1}{\big(\displaystyle\sum_{j=1}^{\infty} \sqrt{a_{1,j}}\big)}=\lambda.
\end{align*}

Note that since for all $n \in \N$, $\sum_{j \in \N} a_{n,j} =a_n$ and $(a_n)_n$ is a bounded sequence, we obtain
\begin{align*}
    \norm{{A}^s(x)}_\infty \lesssim \norm{x}_\infty, \text{ for all } x \in \CM,
\end{align*}
which by definition implies that the sequence $(A^s_n)_n$ is weak type $(\infty, \infty)$. Therefore, by \cite[Theorem 1.3]{cadilhac2025revisiting}, we show that $( A^s_n)_n$ satisfies strong $(p,p)$.

\end{proof}

\begin{thm}\label{thm: bau and su convergence}
    Let $(A_n^h)_n$ be as before. Assume that $1 \leq p_0< p_1 \leq \infty$. Further, assume that $(A^h_n)_n$ is a strong type $(p,p)$ for all $p_0<p<p_1$ and $x \in L^p(\CM)$. Then the following holds.
    \begin{enumerate}[label=(\roman*)]
        \item Let $p_0<p\leq 2p_0$ and $x \in L_p(\CM)$. Then $A^h_n(x)$ converges to $Px$ in b.a.u.
        \item Let $2p_0 <p < p_1$ and $x \in L_p(\CM)$. Then $A^h_n(x)$ converges to $Px$ in a.u.
    \end{enumerate}
\end{thm}

\begin{proof}
    Let us define 
    $$
    S:= span \{ x - \alpha_s(x): x \in \CM \cap L_1(\CM) \}.
    $$
    Choose $x \in S$ such that $x= y - \alpha_t(y)$ for some $y \in \CM \cap L_1(\CM)$. 

    Since $m$ is a right Haar measure and $(h_n)_n$ is a right Reiter sequence, we observe that
    \begin{align*}
        \int_G \alpha_{gt}(y) h_n(g)dm(g)
        = \int_G \alpha_g(y) \rho_{t^{-1}}(h_n)(g)  dm(g).
    \end{align*}
    Hence, 
    \begin{align*}
        A^h_n(x)
        = \int_G \alpha_g(y) \left[ h_n(g) - \rho_{t^{-1}}(h_n)(g) \right] dm(g).
    \end{align*}
    So, we obtain
    \begin{align*}
        \norm{A^h_n(x)}_\infty  \leq \norm{y}_\infty \norm{h_n - \rho_{t^{-1}} (h_n)}_1 \to 0 \text{ as } n \to \infty.
    \end{align*} 
    \emph{Proof of (i):} Let $x \in L_p(\CM)$. Then there exists a sequence $(x_k)$ in $S$ such that 
    \begin{align*}
        \lim_{k \to \infty} \norm{x- Px - x_k}_p = 0,
    \end{align*}
    where $P$ is the projection map from $L_p(\CM)$ onto the fixed point space. Note that $A_n^h P = P A_n^h= P$. So, by strong $(p,p)$ of $(A_n^h)_n$ we get for all $k \in \N$,
    \begin{align*}
        \norm{(A_n^h(x) - Px - A_n^h(x_k))_n}_{L_p(\CM; \ell^\infty)} \lesssim \norm{x- Px - x_k}_p.
    \end{align*}
    Therefore, since $L_p(\CM; c_0)$ is closed in $L_p(\CM; \ell_\infty)$ and $\lim_{k \to \infty} (A_n^h(x_k))_n= (A_n^h(x) - Px)_n$ in $L_p(\CM; \ell^\infty)$, it is now enough to prove that for all $k \in \N$, $(A^h_n(x_k))_n \in L_p(\CM; c_0)$. Again for $x \in S$, just as before we find that $x \in L_q(\CM)$ for some $p_0<q<p$ and since $(A_n^h)_n$ is a strong type $(q,q)$, we have for all $m<n$
    \begin{align*}
    	\norm{(A_j^h(x))_m^n}_{L_p(\CM; \ell^\infty)}
    	& \lesssim \sup_{m \le j \le n} \norm{A_j^h(x)}_\infty^{1-q/p} \norm{(A_j^h(x))_m^n}^{q/p}_{L_q(\CM; \ell^\infty)}.\\
    	& \leq \sup_{m \le j \le n} \left( \norm{y}_\infty \norm{h_n - R_{t^{-1}} h_n}_1  \right)^{1-q/p} \norm{(A_j^h(x))_m^n}^{q/p}_{L_q(\CM; \ell^\infty)}.
    \end{align*}
    Therefore, the finite sequence $(A_1^h(x), \cdots, A_n^h(x),0,0, \cdots)$ is Cauchy and hence converges to $(A_n^h(x))_n$ in $L_p(\CM; \ell^\infty)$. Hence, $(A_n^h(x))_n \in L_p(\CM; c_0)$.\\
    
     \emph{Proof of (ii):} On the other hand, by Kadison's inequality and by the strong type $(p,p)$ inequality, for all $x \in \CM_+ \cap L_p(\CM)_+$, we obtain
     \begin{align*}
     	\norm{(A_n(x))_n}_{L_p(\CM; \ell_\infty^c)} \leq \norm{((A_n(x))^2)_n}_{L_{p/2}(\CM; \ell_\infty)} ^{1/2} \lesssim \norm{x}_p.
     \end{align*}
    Hence, a similar argument as before shows that $(A^h_nx- Px)_n \in L_p(\CM; c^c_0)$.
\end{proof}

\section{Wiener-Wintner Theorem for Reiter sequence}
This section is devoted to a non-commutative Wiener-Wintner theorem along a Reiter sequence for a trace-preserving dynamical system $(\CM, G, \alpha)$ where $\CM$ is finite von Neumann algebra equipped with faithful normal trace $\tau$ and $G$ be a locally compact second countable group that admits a two-sided Reiter sequence.
Now any l.c.s.c group that admits a two-sided Reiter sequence is unimodular. Indeed, suppose $(s_n)_n$ is a two-sided Reiter sequence and $m$ be a left invariant Haar measure. The right Reiter condition yields
$\|\rho_g(s_n)-s_n\|_1 \to 0$ for all $g \in G.$
Observe that \begin{align*}
\|\rho_g(s_n)\|_1=\int_Gs_n(xg)dm(x)=\int_G\Delta(g)^{-1}s_n(x)dm(x)=\Delta(g)^{-1},
\end{align*}
where $\Delta:G \to (0,\infty)$ be the modular function.  Therefore, 
\begin{align*}
    |\Delta(g)^{-1}-1|=|\|\rho_g(s_n)\|_1-\|s_n\|_1|\leq \|\rho_g(s_n)-s_n\|_1\rightarrow 0.
\end{align*}
This implies $\Delta(g)\equiv 1,$ which ensures that $G$ is unimodular. 
Let $\pi:G \to \CU_d(\mathbb{C})$ be a unitary representation of $G$ where $\CU_d(\mathbb{C})$ is group of all unitary $d\times d$-matrices, $\phi \in C(\CU_d(\mathbb{C}))$ with $\|\phi\|_\infty=1$ and $x \in L^p(\CM,\tau).$  Let us define the weighted average along a Reiter sequence $(s_n)$
$${A}_n^{\pi,\phi,u}(x):=\int_G\alpha_g(x)\phi(\pi(g)u)s_n(g)dm(g), ~~~u\in \CU_d(\mathbb{C}).$$
Our goal is to prove the b.a.u. convergence of ${A}_n^{\pi,\phi,u}(x)$ for $x \in L^p(\CM,\tau)$ with $1 \leq p < \infty.$ Crucially, the convergence is uniform over the symbols $(\pi, \phi, u)$, that is, the existing projection required for b.a.u. convergence does not depend on these symbols. In this section, we will prove the following theorem.

\begin{thm}\label{WW_theorem}
Let $G$ be a l.c.s.c group with a two-sided Reiter sequence acting continuously on a finite von Neumann algebra $\CM$ with finite trace $\tau.$ Suppose that $x\in L^p(\CM, \tau)$ and $(\pi, \phi, u)$ is as above. Then ${A}_n^{\pi,\phi,u}(x)$ converges in b.a.u. uniformly over the symbol $(\pi, \phi, u).$ Moreover, if $x$ is in the weak-mixing part of the representation of $\alpha$, then the average converges to $0$ in b.a.u. 
\end{thm}

We prove this theorem in multiple steps. 
First, we derive the following theorem from Theorem 
\ref{weak_type_(1,1)_ergodic}, which proves a weak type-$(1,1)$ maximal inequality for the above weighted average ${A}_n^{\pi,\phi,u}(x)$. Since the proof uses techniques similar to \cite[Theorem 4.5]{BGK_WW_2026}, we omit the details. 
\begin{thm}\label{weak_1,1}
    Let $x\in L^1(\CM,\tau)$ and $\lambda >0.$ Then there exists a projection $e_x \in \CM$ such that 
    \begin{align*}
        \sup_n\|e_x{A}_n^{\pi,\phi,u}(x)e_x\|_\infty \leq \lambda ~\text{ and } ~\tau(1-e_x) \lesssim \frac{\|x\|_1}{\lambda}. 
    \end{align*}
\end{thm}
\subsection{\textbf{JdLG decomposition}}
Establishing the main convergence result requires splitting the space $L^2(\CM, \tau)$ into a weakly mixing subspace $L^2(\CM, \tau)_{wm}$ and a compact subspace $L^2(\CM, \tau)_c$, which is known as JdLG decomposition. In October 2025, Sohail Farhangi,Y.Kuznetsova and Mickey Barthmann proved the following version of JdLG decomposition for the amenable group action on a Banach space.
\begin{thm}\label{JdlG_decomposition}\cite{barthmann-farhangi-kuznetsova2025jacobsdeleeuwglicksbergdecomposition}
Let $G$ be a locally compact amenable group, $m$ a left Haar measure,
$ (F_i)_{i \in I}$ a left F\o lner net, $E$ a Banach space, and
$\pi$ a relatively weakly compact representation of $G$ on $E$.
Then $\pi$ is JdLG-admissible, and it follows that 
\begin{align*}
E_c :&= \{\, \xi \in E : \pi(G)\xi \text{ is relatively compact in } E \,\}  \text{ and }\\
E_{wm} :&= \left\{\, \xi \in E :
\lim_{i} \frac{1}{m(F_i)}
\int_{F_i} \bigl| \langle \pi(g)\xi, x' \rangle \bigr| \, dm(g) = 0
\text{ for all } x' \in E' \right\}.
\end{align*}
\end{thm}
$E_c$ is referred to as the \emph{ compact } part and  $ E_{wm}$ is referred to as the \emph{weak mixing} part.
 When  $E=\CH$, a Hilbert space, and $\pi$ is a unitary representation of $G$ on $\CH$. Since $\| \pi(g) \|=1$ for all $g \in G$, and the closed unit ball of $\CB(\CH)$ under WOT is compact, and the representation is JdLG admissible by  \cite[Example 16.25, Chapter 16]{Eisner2015}. Moreover, we have $ \CH= \CH_c \oplus \CH_{wm}$.\\
Another version of the JdlG decomposition shows that $\CM$ decomposes into compact and weakly mixing parts, denoted by $\CM_c$ and $\CM_{wm}$ respectively, given by
$$\CM_c = \CM \cap L^2(\CM,\tau)_c \quad \text{and} \quad \CM_{wm} = \CM \cap L^2(\CM,\tau)_{wm}.$$ We recall the following theorem from \cite{BGK_WW_2026}.
    \begin{thm}\label{JdLG}
Let $ \CM$ be a finite von Neumann algebra with a f.n trace $\tau$ and $ (\CM, G, \alpha) $ be a $\tau$-preserving non-commutative dynamical system. Then $\CM$ can be written as 
$\CM=\CM_{c}\oplus \CM_{wm}$ such that 
\[ \overline{ \CM \cap L^{2}_{c}(\CM,\tau)} = L_{c}^{2}(\CM,\tau)    ~\text{ and }~  \overline{ \CM \cap L^{2}_{\text{wm}}(\CM,\tau) } ~=  L^{2}_{\text{wm}}(\CM,\tau) .\]
\end{thm}

One of the most useful results in classical and non-commutative ergodic theory is the van der Corput inequality. Although various versions exist in the literature, a non-commutative analogue was first proved in \cite{NSZJOT2003}. For our purpose we need the following version of the inequality. 
\begin{lem}[van der Corput inequality along a Reiter sequence]\label{vand_corput}
    Let $f: G \to \CM$ be a bounded measurable function and $(s_n)$ be a right Reiter sequence. Then we have 
    \begin{align}\label{vand_cor_inequality}
        \bigg\|\int_G f(g)s_n(g)dm(g)\bigg\|^2 &\leq \bigg\|\int_G\int_G\bigg(\int_Gf(g)^*f(gh_1^{-1}h_2)s_n(g)dm(g)\bigg)s_k(h_1)s_k(h_2)dm(h_1)dm(h_2)\bigg\|\nonumber\\ &+ 3\|f\|_\infty^2 \int_G\|s_n -\rho_{h^{-1}}(s_n)\|_1s_k(h)dm(h) + \|f\|_\infty^2 \Bigg(\int_G\|s_n -\rho_{h^{-1}}(s_n)\|_1s_k(h)dm(h)\Bigg)^2.
    \end{align}
\end{lem}
\begin{proof}
We begin by evaluating the following,
\begin{align*}
    \bigg\|\int_G f(g)s_n(g)dm(g)-&\int_G\int_Gf(gh)s_n(g)s_k(h)dm(g)dm(h)\bigg\| \\
    &\leq \bigg\|\int_G \int_G f(g)s_n(g)s_k(h)dm(g)dm(h)-\int_G\int_Gf(g)s_n(gh^{-1})s_k(h)dm(g)dm(h)\bigg\|\\
    &\leq \bigg\|\int_G\bigg(\int_Gf(g)s_n(g)s_k(h)dm(g) - \int_G f(g)s_n(gh^{-1})s_k(h)dm(g)\bigg)dm(h)\bigg\|\\
   &\leq \int_G\int_G \|f\|_\infty s_k(h)|s_n(g) - s_n(gh^{-1})|dm(g)dm(h)\\
   &= \|f\|_\infty\int_G  \|s_n-\rho_{h^{-1}} (s_n)\|_1s_k(h)dm(h).
\end{align*}
Therefore, we obtain 
\begin{align}\label{A}
    \bigg\|\int_G f(g)s_n(g)dm(g)\bigg\|^2 \leq \bigg\|\int_G\int_G&f(gh)s_n(g)s_k(h)dm(g)dm(h)\bigg\|^2+ 2\|f\|_\infty^2\int_G  \|s_n-\rho_{h^{-1}} (s_n)\|_1s_k(h)dm(h)\nonumber\\
    &+ \|f\|_\infty^2 \Bigg(\int_G  \|s_n-\rho_{h^{-1}} (s_n)\|_1s_k(h)dm(h)\Bigg)^2.
\end{align}
We now estimate the first term separately, 
\begin{align*}
&\bigg\|\int_G\int_Gf(gh)s_n(g)s_k(h)dm(g)dm(h)\bigg\|^2 \\
  = & \bigg\|\int_G\bigg(\int_Gf(gh)s_k(h)dm(h)\bigg)s_n(g)dm(g)\bigg\|^2 \\
  &\leq \bigg\|\int_G \bigg|\int_Gf(gh)s_k(h)dm(h) \bigg|^2s_n(g)dm(g)\bigg\|\\
&\leq\bigg\|\int_G\int_G\int_Gf(gh_1)^*f(gh_2)s_k(h_1)s_k(h_2)dm(h_1)dm(h_2)s_n(g)dm(g)\bigg\|\\
&=\bigg\|\int_G\int_G\bigg(\int_Gf(gh_1)^*f(gh_2)s_n(g)dm(g)\bigg)s_k(h_1)s_k(h_2)dm(h_1)dm(h_2)\bigg\|\\
&=\bigg\|\int_G\int_G\bigg(\int_Gf(g)^*f(gh_1^{-1}h_2)s_n(gh_1^{-1})dm(g)\bigg)s_k(h_1)s_k(h_2)dm(h_1)dm(h_2)\bigg\|\\
&\leq\bigg\|\int_G\int_G\bigg(\int_Gf(g)^*f(gh_1^{-1}h_2)s_n(g)dm(g)\bigg)s_k(h_1)s_k(h_2)dm(h_1)dm(h_2)\bigg\|\\
&\qquad + \|f\|_{\infty}^2\int_G\|s_n-\rho_{h_1^{-1}}(s_n)\|s_k(h_1)dm(h_1). 
\end{align*}
Hence, the lemma follows from Ineq. \ref{A} and the above estimate. 
\end{proof}
By Theorem \ref{weak_1,1}, the weighted average ${A}_n^{\pi,\phi,u}$ is of weak type $(1,1)$, while its strong type $(\infty,\infty)$ bound is immediate. As a consequence, it follows that ${A}_n^{\pi,\phi,u}$ is strong $(p,p)$ for all $1 < p < \infty$ by interpolation. Now the next theorem is Banach Principle, which is analogous to its classical counterpart which provides a link between the maximal inequality and the b.a.u. convergence. We refer to \cite[Theorem 2.1]{Lit2014} for the proof. 
\begin{thm}\label{Banach_Principle}
    For $1 \leq p < \infty$ the following set 
    $$\mathcal{C}=\{x \in L^p(\CM, \tau):{A}_n^{\pi,\phi,u}(x) \text{ converges in b.a.u. uniformly in }(\pi, \phi, u) \}$$
    is closed in $L^p(\CM,\tau).$
\end{thm}
The next proposition is a version of \cite[Proposition 4.3]{BGK_WW_2026} stated for Reiter sequences instead of F\o lner sequences. Since the argument is analogous, we omit the proof.
\begin{prop}\label{ergodic_L2}
    Let $(\CM, G, \alpha)$ be as above and the action $\alpha$ is ergodic (i.e. $\CM^G=\C1$). Suppose $(h_n)_n$ is a left Reiter sequence and $m$ is a left Haar measure. Then for every $x\in L^2(\CM,\tau)$,
    $$A^h_n(x) \to \tau(x)1$$ almost uniformly (a.u.) and in $\|\cdot\|_2$.
\end{prop}

\begin{rem}\label{simplification}
    Since $\CM\cap L^2(\CM,\tau)$ is dense in $L^p(\CM,\tau)$ for all $1<p\leq \infty,$ by Theorem \ref{Banach_Principle} it suffices to prove the b.a.u. convergence of ${A}_n^{\pi,\phi,u}$ for $x \in L^2(\CM, \tau).$ Using the JdLG decomposition (Theorem \ref{JdLG}), we will prove this convergence separately for $x \in L^2(\CM,\tau)_c$ and $x \in \CM_{wm}$, which will complete the proof of Theorem \ref{WW_theorem}. Since $\CU_d$ is a compact group, the Stone-Weierstrass theorem allows us to approximate $\phi$ uniformly by matrix coefficients, that is the functions of the form $u\in \CU_d \mapsto \langle u\xi,\eta\rangle,~~\xi,\eta \in \C^d.$ Therefore, to prove convergence, it suffices to assume that $\phi$ is of this form. 
\end{rem}
The following theorem proves the convergence of the weighted average ${A}_n^{\pi,\phi,u}(x),$ when $x$ is from the compact part. 
\begin{thm}
   Let $x\in L^2(\CM,\tau)_c$ and $\xi, \eta \in \C.$ Then the following weighted average 
   $$B_n(x)= \int_G\alpha_g(x)\langle \pi(g)u\xi,\eta \rangle s_n(g)dm(g)$$ converges in a.u.
\end{thm}
\begin{proof}
In  view of \cite[Theorem 15.14]{Eisner2015} and Theorem \ref{Banach_Principle}, it suffices to prove the convergence of $B_n(x)$ for $x$ from finite dimensional $G-$ invariant subspace of $L^2(\CM, \tau).$ Let $\mathcal{K}$ be a finite dimensional invariant subspace of $L^2(\CM, \tau)$ with an orthonormal basis $\{x_1,\cdots, x_m\}.$ Then we have 
$$\alpha_g(x)=\sum_{i=1}^m \langle\alpha_g(x),x_i\rangle x_i$$
Now 
\begin{align*}
    B_n(x)&=\sum_{i=1}^m\int_G\langle\alpha_g(x),x_i
\rangle \langle\pi(g)u\xi,\eta \rangle s_n(g)dm(g) x_i\\
&=\sum_{i=1}^m\int_G\bigg\langle\alpha_g(x)\otimes \pi(g)u\xi,x_i\otimes \eta\bigg\rangle  s_n(g)dm(g) x_i\\
&=\sum_{i=1}^m\bigg\langle\int_G\alpha_g(x)\otimes \pi(g)u\xi ~s_n(g) dm(g), x_i\otimes \eta \bigg\rangle  x_i.
\end{align*}
Applying the Mean Ergodic Theorem for Reiter sequence yields
$$\int_G\alpha_g(x)\otimes \pi(g)u\xi ~s_n(g) dm(g)= \int_G(\alpha_g\otimes \pi(g))(x\otimes u\xi)s_n(g)dm(g)$$
converges. Therefore, since there are finitely many $x_i$ we can find a single projection $e \in \CM$ with arbitrarily small trace of its complement such that $\mathbb{B}_n(x)e$ converges.  
\end{proof}
Next, we present a theorem on weak-mixing criteria along a Reiter sequence from \cite[Theorem 3.13]{BGK_WW_2026}. 
\begin{thm}
    Let $G$ be an amenable locally compact group, $\pi \colon G \to \mathcal{U}(\mathcal{H})$ a strongly continuous unitary representation, and $(\mu_n)$ a left Reiter sequence on $G$. If $\xi \in \mathcal{H}_{\mathrm{wm}}$, then
\begin{align*}
    \int_G |\langle \pi(g)\xi, \eta \rangle| \, d\mu_n(g) \longrightarrow 0 \quad \text{for all } \eta \in \mathcal{H}.
\end{align*}
\end{thm}
The following lemma is needed to prove the convergence of the weak mixing part. Since the techniques of the proof is similar to the proof of \cite[Lemma 4.13]{BGK_WW_2026}, we omit the details. 
\begin{lem}\label{keylemma}
    Let $G$ be a l.c.s.c group with two-sided Reiter sequence $(s_n)_n$ and $(\CM, \tau)$ be a finite von Neumann algebra. $\alpha: G \to Aut (\CM)$ be a trace preserving weakly mixing action. Suppose $x \in \CM$ with $\tau(x)=0$, Then we have 
    \begin{align*}
        \int_G\int_G |\tau(x^*\alpha_{h_1^{-1}h_2}(x))|s_k(h_1)s_k(h_2)dm(h_1)dm(h_2) \longrightarrow 0.
    \end{align*}
\end{lem}

\begin{thm}
    Let $G$ be a l.c.s.c group with two-sided Reiter sequence $(s_n)$ and $\CM$ be a finite von Neumann algebra with f.n trace $\tau.$ Suppose $(\CM,G, \alpha)$ is as before. Let $x\in \CM_{wm}.$ Then ${A}_n^{\pi,\phi,u}(x)$ converges to $0$ in a.u. uniformly in symbols $(\pi,\phi,u).$
\end{thm}
\begin{proof}
    In view of Remark \ref{simplification}, we can take the weighted average as
$${A}_n^{\pi,\phi,u}(x)= \int_G\alpha_g(x)\langle \pi(g)u\xi,\eta \rangle s_n(g)dm(g), ~~\xi,\eta \in \C^d.$$
Since $G$ is second countable, there exists a countable dense subset $\{g_i\}_{i \in \N}$ of $G.$
Let $\varepsilon >0.$ Applying Proposition \ref{ergodic_L2} there exist projections $e_i \in \CM$ with $\tau(1-e_i) <\frac{\varepsilon}{2^i}$ and 
$$A^s_n(x^*\alpha_{g_i}(x))e_i \to \tau(x^*\alpha_{g_i}(x))e_i \text{ in } \|\cdot\|_{\infty}.$$ Define $e:=\displaystyle \bigwedge_{i\in \N} e_i.$ Therefore $$A^h_n(x^*\alpha_{g_i}(x))e \to \tau(x^*\alpha_{g_i}(x))e \text{ in } \|\cdot\|_{\infty} \text{ and } \tau(1-e) < \varepsilon.$$
Since the action $\alpha$ is continuous and $\{g_i\}_{i\in \N}$ is dense in $G$, it follows that 
\begin{align*}
    A^h_n(x^*\alpha_{g}(x))e \to \tau(x^*\alpha_{g}(x))e \text{ in } \|\cdot\|_{\infty} \text{ for all } g \in G.
\end{align*}
Consider the function $f:G \to \CM \otimes M_d(\C),$ defined by $f(g):= \alpha_g(x)e\otimes \pi(g)u,~~u\in \CU_d.$ So, $f$ is bounded under $\|\cdot\|_\infty$ of $\CM \otimes M_d(\C).$ Next we will use the Lemma \ref{vand_corput}. Since $(s_n)_n$ is a two-sided Reiter sequence and $s_k\in L^1(G)$, by the dominated convergence theorem the last two terms on the R.H.S of Ineq. \ref{vand_cor_inequality} converge to $0$. Hence, it suffices to prove that the first term on the R.H.S of Ineq. \ref{vand_cor_inequality} converges to $0$. We have
$$f(g)^*f(gh_1^{-1}h_2)=e\alpha_g(x^*\alpha_{h_1^{-1}h_2}(x))e\otimes u^*\pi(h_1^{-1}h_2)u.$$
Therefore, computing the first term of the R.H.S of Ineq. \ref{vand_cor_inequality} we obtain
\begin{align*}
\bigg\|&\int_G\int_G\bigg(\int_Gf(g)^*f(gh_1^{-1}h_2)s_n(g)dm(g)\bigg)s_k(h_1)s_k(h_2)dm(h_1)dm(h_2)\bigg\|\\
&=\bigg\|\int_G\int_G\bigg(\int_Ge\alpha_g(x^*\alpha_{h_1^{-1}h_2}(x))e\otimes u^*\pi(h_1^{-1}h_2)u~s_n(g)dm(g)\bigg)s_k(h_1)s_k(h_2)dm(h_1)dm(h_2)\bigg\|\\
&=\int_G\int_G\bigg\|\int_Ge\alpha_g(x^*\alpha_{h_1^{-1}h_2}(x))e\otimes u^*\pi(h_1^{-1}h_2)u~s_n(g)dm(g)\bigg\|s_k(h_1)s_k(h_2)dm(h_1)dm(h_2)\\
&\leq \int_G\int_G\bigg\|{A}^s_n(x^*\alpha_{h_1^{-1}h_2}(x))\bigg\|s_k(h_1)s_k(h_2)dm(h_1)dm(h_2)\\
&= \int_G\int_G\big|\tau(x^*\alpha_{h_1^{-1}h_2}(x))\big|s_k(h_1)s_k(h_2)dm(h_1)dm(h_2).
\end{align*}
Hence, using Lemma \ref{keylemma} we obtain
$$ \int_G\big(\alpha_g(x)e \otimes \pi(g)u\big) ~s_n(g)dm(g) \longrightarrow 0$$ in $\|\cdot\|_\infty$ of the von Neumann algebra $\CM\otimes M_d(\C).$ Therefore, 
\begin{align*}
    {A}_n^{\pi,\phi,u}(x)e&= \int_G\alpha_g(x)e\langle \pi(g)u\xi,\eta \rangle s_n(g)dm(g)\\
    &=\int_Gid_{\CM}\otimes\omega_{\xi,\eta}(\alpha_g(x)e\otimes \pi(g)u)~s_n(g)dm(g)\\
    &=id_{\CM}\otimes\omega_{\xi,\eta} \bigg(\int_G\big(\alpha_g(x)e \otimes \pi(g)u\big) ~s_n(g)dm(g)\bigg),
\end{align*}
where $\omega_{\xi,\eta}:\CU_d \to \C$ defined as $\omega_{\xi,\eta}(u):= \langle u\xi,\eta\rangle.$ Since $id_{\CM}\otimes \omega_{\xi,\eta}$ is a bounded operator, we conclude that ${A}_n^{\pi,\phi,u}(x)e \longrightarrow 0$ in $\|\cdot\|_{\infty}$ of $\CM.$
\end{proof}

\medskip
\noindent\textbf{Acknowledgments.} 
First author P. Bikram gratefully acknowledges financial support provided by the Anusandhan National Research Foundation (ANRF), Government of India, under the : ANRF/ARGM/2025/001021/MTR  grant.



    \newcommand{\etalchar}[1]{$^{#1}$}
\providecommand{\bysame}{\leavevmode\hbox to3em{\hrulefill}\thinspace}
\providecommand{\MR}{\relax\ifhmode\unskip\space\fi MR }
\providecommand{\MRhref}[2]{%
  \href{http://www.ams.org/mathscinet-getitem?mr=#1}{#2}
}
\providecommand{\href}[2]{#2}

\end{document}